\documentclass[a4paper,11pt,intlimits,oneside]{amsart}

\usepackage{amssymb,latexsym,amsmath,mathrsfs}
\usepackage{amsfonts,amsbsy,bm}
\usepackage{color}
\usepackage[margin=1.5cm]{geometry}

\usepackage{amsmath}
\usepackage{amssymb}
\usepackage{amsfonts}
\usepackage{amsthm,amscd}

\newtheorem{theorem}{Theorem}[section]
\newtheorem{proposition}[theorem]{Proposition}
\newtheorem{corollary}[theorem]{Corollary}
\newtheorem{lemma}[theorem]{Lemma}

\newtheorem{remark}[theorem]{Remark}

\theoremstyle{definition}
\newcommand{\comment}[1]{}

\numberwithin{equation}{section}

\newcommand{\epf}{ $\Box$\medskip}

\theoremstyle{definition}

\begin{document}
\title[J.M. Tanoh Dje and Beno\^it. F.  Sehba]{Atomic decomposition of Bergman-Orlicz space on the upper complex half-plane}
\author[J.M. Tanoh Dje and Beno\^it. F.  Sehba]{Jean$-$Marcel Tanoh Dje and Beno\^it F. Sehba}
\address{Unit\'e de Recherche et d'Expertise Num\'erique, Universit\'e Virtuelle de C\^ote d'Ivoire, Cocody II-Plateaux - 28 BP 536 ABIDJAN 28}
\email{{\tt tanoh.dje@uvci.edu.ci}}
\address{Department of Mathematics, University of Ghana,  P.O. Box L.G 62 Legon, Accra, Ghana.}
\email{{\tt bfsehba@ug.edu.gh}}


\date{}

\begin{abstract}
In this work, we propose an atomic decomposition of the Bergman-Orlicz spaces on the complex upper half-plane. Using this result, we characterize Carleson embeddings with loss between Bergman-Orlicz spaces and certain Orlicz spaces. We also leverage this last result to control the composition operator between two Bergman-Orlicz spaces.
\end{abstract}

\maketitle

\section{Introduction and results.}

In this paper, a function $\Phi : [0,\infty) \rightarrow [0,\infty)$ is called a growth function if it is nondecreasing, $\lim_{t \to 0} \Phi(t) =\Phi(0)= 0$, $\Phi(t) > 0$
for $t \in (0,\infty)$ and $\lim_{t \to \infty} \Phi(t) = \infty$.  The growth function  $\Phi$ is said to be of upper type (resp. lower type) if there exists  $p \in (0,\infty)$ and a constant $C>0$ such that for all $t \in  [1,\infty)$ (resp. $t \in  [0,1]$) and $s \in  [0,\infty)$,
\begin{equation}\label{eq:sui8n}
\Phi(st)\leq Ct^{p}\Phi(s).\end{equation}
We denote by $\mathscr{U}^{p}$ (resp. $\mathscr{L}_{p}$) the set of all growth functions of upper-type $p \geq 1$ (resp. lower-type $0< p\leq 1$) such that the function 
$t\mapsto \frac{\Phi(t)}{t}$ is non decreasing (resp. non-increasing) on $(0,\infty)$. We put
$   \mathscr{U}:=\bigcup_{p\geq 1}\mathscr{U}^{p}$ (resp. $\mathscr{L}:=\bigcup_{0< p\leq 1}\mathscr{L}_{p}$).

\medskip

Let  $\Phi$ be a growth function. We say that  $\Phi$ satisfies the $\Delta_{2}-$condition (or $\Phi \in \Delta_{2}$) if there exists a constant
$K > 1$ such that
\begin{equation}\label{eq:delta2}
\Phi(2t) \leq K \Phi(t),~ \forall~ t > 0.\end{equation}
When $\Phi$ is convex, we say that $\Phi$ satisfies $\nabla_{2}-$condition (or $\Phi \in \nabla_{2}$)  if $\Phi$ and its complementary function both satisfy $\Delta_{2}-$condition. Note that the complementary function of $\Phi$ is the function $\Psi$ defined by
$$ \Psi(s)=\sup_{t\geq 0}\{st-\Phi(t) \}, ~ \forall~  s \geq 0.       $$

\medskip

We say that two growth functions $\Phi_{1}$ and $\Phi_{2}$ are equivalent (or $\Phi_{1} \sim \Phi_{2}$), if there exists a constant $c > 0$ such that
\begin{equation}\label{eq:equivalent}
c^{-1}\Phi_{1}(c^{-1}t) \leq \Phi_{2}(t)\leq c\Phi_{1}(ct), ~~ \forall~ t > 0.\end{equation} 
We will assume in the sequel that  any element $\Phi$ of $\mathscr{U}$ (resp. $\mathscr{L}$) belongs to $\mathscr{C}^{1}([0,\infty))$ and is  
convex (resp. concave),
(see for example \cite{bonamisehba,  djesehb, djesehaqb, sehbaedgc}).

\medskip

Let $(X, \sum, \mu)$ be a measure space and $\Phi$ a  growth function  of lower type. The Orlicz space on  $X$, $L^{\Phi}(X, d\mu)$ is the space of all equivalent
classes (in the usual sense) of measurable functions $f : X \longrightarrow \mathbb{C}$ which satisfy
$$ \|f\|_{L_{\mu}^{\Phi}}^{lux}:=\inf\left\{\lambda>0 : \int_{X}\Phi\left(\dfrac{| f(x)|}{\lambda}\right)d\mu(x) \leq 1  \right\}< \infty. $$
The space  $L^{\Phi}$  generalizes the Lebesgue space $L^{p}$  for $0< p < \infty$  (see \cite{shuchen, kokokrbec, raoren}).

\medskip

Let  $\mathbb{C}_{+}:=\left\{ z=x+iy \in \mathbb{C} :   y > 0 \right\}$ be the upper complex half$-$plane. Let  $\alpha > -1$ be a real and $\Phi$ a  growth function of lower-type. 
\begin{itemize}
\item[\textbullet] The Bergman$-$Orlicz space on  $\mathbb{C_{+}}$, $A_{\alpha}^{\Phi}(\mathbb{C_{+}})$ is the subspace of  $L^{\Phi}(\mathbb{C_{+}}, dV_{\alpha})$ constituted of analytic functions on $\mathbb{C_{+}}$, here $dV_\alpha(z):=(\Im mz)^\alpha dV(z)$, and $dV(z):=dxdy$ if $z=x+iy$. For $F \in A^{\Phi}_{\alpha}(\mathbb{C_{+}})$, we will simply use the notation $\|F\|_{A^{\Phi}_{\alpha}}^{lux}:= \|F\|_{L_{V_{\alpha}}^{\Phi}}^{lux}$.
\item[\textbullet] The Orlicz space of sequences on  $\mathbb{C}_{+}$,  $\ell_{\alpha}^{\Phi}(\mathbb{C}_{+})$ is the space of all sequences of complex numbers   $\mu=\{\mu_{l,j}\}_{l,j\in \mathbb{Z}}$ on  $\mathbb{C_{+}}$ such that 
$$  \|\mu\|_{\ell^{\Phi}_{\alpha}}^{lux}:= \inf\left\{\lambda>0 : \sum_{j\in \mathbb{Z}}\sum_{l\in \mathbb{Z}}\Phi\left(\dfrac{| \mu_{l,j}|}{\lambda}\right)2^{j\gamma (\alpha +2)}  \leq 1  \right\}< \infty   $$
where   $\gamma$ is a real number such that
 \begin{equation}\label{eq:uaa5aqn}
 \frac{\log\left( \frac{1+\delta^{2}/20}{1-\delta^{2}/20}  \right)}{4 \log 2}< \gamma  < \frac{\log\left( \frac{1+\delta^{2}/4}{1-\delta^{2}/4}  \right)}{4 \log 2},
\end{equation}
with   $0<\delta< 1$.
\end{itemize}
The space $A_{\alpha}^{\Phi}(\mathbb{C_{+}})$ (resp. $\ell_{\alpha}^{\Phi}(\mathbb{C_{+}})$) generalizes  the Bergman space $A_{\alpha}^{p}(\mathbb{C_{+}})$ (resp.  sequence space $\ell_{\alpha}^{p}(\mathbb{C_{+}})$), for $0< p < \infty$ that correspond to the case $\Phi(t)=t^p$. Note that for $F$ to be in the Bergman space $A_{\alpha}^{p}(\mathbb{C_{+}})$, we should have that
$$\|F\|_{\alpha,p}:=\left(\int_{\mathbb{C}_+}\vert F(z)\vert^pdV_\alpha(z)\right)^{\frac 1p}<\infty.$$
We are interested in this paper in the atomic decomposition of Bergman–Orlicz spaces, that is the representation of each function in these spaces as a linear combination of the reproducing kernels. Atomic decomposition of classical Bergman spaces has a long history and has been intensively studied in the case of the unit disc, the unit ball of $\mathbb{C}^n$ and the upper half-plane (see \cite{gonessa,RT,zhu} and the references therein) with applications in the study of operators such as Hankel operators and Toeplitz operators (see \cite{pauzhao,tchoundja}), in the study of Carleson measures (see \cite{Luecking1,Luecking2,Luecking3,Luecking4,sehba4}), and in weak factorization problems. Atomic decomposition for Bergman–Orlicz spaces of the unit ball was recently provided in \cite{bekollebonamitchoundja} with applications to weak factorization of some of these spaces. The second author of the present paper obtained in \cite{sehba4}, using the atomic decomposition from \cite{bekollebonamitchoundja}, a characterization of Carleson embedding with loss for the Bergman–Orlicz spaces of the unit ball.
\medskip

The aim of this paper is to develop and atomic decomposition of Bergman-Orlicz spaces of the upper half-plane and provide some applications.
\medskip 

For $0<\delta< 1$, we call $\delta-$lattice, the sequence of complex numbers $\{z_{l,j}\}_{l,j\in \mathbb{Z}}$ defined by   
\begin{equation}\label{eq:ua5n}
z_{l,j}:= \delta^{2}l2^{\gamma j-3}+i2^{\gamma j},~~ \forall~l,j\in \mathbb{Z},
\end{equation}
where $\gamma$ is the real defined by (\ref{eq:uaa5aqn}). More on this sequence will be said in Section 3 (see also \cite{gonessa}).
\vskip .2cm
Our result on the atomic decomposition of the Bergman-Orlicz spaces is as follows.

\begin{theorem}\label{pro:mainqaaq6}
Let $\alpha>-1$,  $0<\delta<1$ and  $\Phi \in \mathscr{U} \cap \nabla_{2}$. The following assertions are satisfied.
\begin{itemize}
\item[i)] For   $\mu:=\{\mu_{l,j}\}_{l,j\in \mathbb{Z}} \in \ell^{\Phi}_{\alpha}(\mathbb{C_{+}})$, the function defined by
$$  F_{\mu}(z) =c_{\alpha}\sum_{j\in \mathbb{Z}}\sum_{l\in \mathbb{Z}}\mu_{l,j}\left(  \frac{z-\overline{z_{l,j}}}{i} \right)^{-\alpha-2}2^{j\gamma (\alpha +2)},~~ \forall~z \in \mathbb{C}_{+},
   $$
belongs to  $A_{\alpha}^{\Phi}(\mathbb{C}_{+})$, where $c_{\alpha}$ is a constant which depends only on $\alpha$. Moreover,  there exists a constant $C_{1}:= C_{\Phi,\delta,\alpha}>0$ such that  
\begin{equation}\label{eq:ineagay}
\|F_{\mu}\|_{A^{\Phi}_{\alpha}}^{lux}\leq C_{1} \left\|\mu\right\|_{\ell^{\Phi}_{\alpha}}^{lux}.
 \end{equation}
\item[ii)] For  $F\in A_{\alpha}^{\Phi}(\mathbb{C}_{+})$and for $\delta$  small enough,  there exists a 
 sequence $\mu:=\{\mu_{l,j}\}_{l,j\in \mathbb{Z}}$ belonging to $\ell^{\Phi}_{\alpha}(\mathbb{C_{+}})$ such that 
$$  F(z) =c_{\alpha}\sum_{j\in \mathbb{Z}}\sum_{l\in \mathbb{Z}}\mu_{l,j}\left(  \frac{z-\overline{z_{l,j}}}{i} \right)^{-\alpha-2}2^{j\gamma (\alpha +2)},~~ \forall~z \in \mathbb{C}_{+},
   $$
where $c_{\alpha}$ is a constant which depends only on $\alpha$.  Moreover, there exist constants $C_2,C_2'>0$, depending only
on $\Phi$, $\delta$, and $\alpha$, such that
\begin{equation}\label{eq:inaqaqay}
C_2\|\mu\|_{\ell_\alpha^\Phi}^{\mathrm{lux}}
\leq
\|F\|_{A_\alpha^\Phi}^{\mathrm{lux}}
\leq
C_2'\|\mu\|_{\ell_\alpha^\Phi}^{\mathrm{lux}}.
\end{equation}
\end{itemize}
\end{theorem} 

In proving the above result, instead of using Bergman's balls as it is usually done in the literature, we work with Euclidean disks that lie in upper half-plane. The sequence space is indexed by a concrete lattice, so comparing continuous and discrete Orlicz sums becomes a straightforward geometric statement. But this choice of Euclidean disks requires more detailed local analysis than the bounded case. Indeed, instead of global estimates on a homogeneous domain, we need local estimates on each disk, and we have to explicitly control the oscillations of an analytic function on the lattice cells. The $\delta$-lattice is chosen so that the associated disks have finite overlap and contain pairwise disjoint subdisks, allowing us to compare continuous weighted integrals with discrete Orlicz sums. The main tools are local pointwise estimates for $\Phi(|F|)$ on Euclidean disks, obtained via subharmonicity, Jensen's inequality and the convexity of the associated growth function; oscillation estimates for $F$ on the cells of the lattice, obtained via Cauchy's integral formula; and the sampling lemma which compares $\|F\|_{A_\alpha^\Phi}^{\mathrm{lux}}$ with $\|\{F(z_{l,j})\}\|_{\ell_\alpha^\Phi}^{\mathrm{lux}}$.
\medskip

In the previous works \cite{djesehb, djesehafeuto}, we provided necessary and sufficient conditions for a Bergman-Orlicz space $A_{\alpha}^{\Phi_{1}}$  to be continuously embedded into the Orlicz space $L^{\Phi_{2}}(\mathbb{C}_{+}, d\mu)$ (Carleson embedding),
 under the condition that the function   $t\mapsto \frac{\Phi_{2}(t)}{\Phi_{1}(t)}$  is increasing on  $(0, \infty)$ (for the unit ball of $\mathbb{C}^n$, see  \cite{CharpentierSehba,sehba5} and the references therein).  When 
 $\Phi_{1}(t)=t^{p}$ and $\Phi_{2}(t)=t^{q}$, with $0<p, q<\infty$, the condition that  $t\mapsto \frac{\Phi_{2}(t)}{\Phi_{1}(t)}$ is increasing is equivalent to  $p\leq q$. We notice that the generalization of the case $q<p$ has not been addressed. In this work, we use the above atomic decomposition of Bergman-Orlicz spaces to investigate this open case. More precisely, we formulate necessary and sufficient conditions for the Bergman-Orlicz space $A_{\alpha}^{\Phi_{1}}$ to be continuously embedded into an Orlicz space  $L^{\Phi_{2}}(\mathbb{C}_{+}, d\mu)$, this time under the condition that   $t\mapsto \frac{\Phi_{1}(t)}{\Phi_{2}(t)}$  is increasing on  $(0, \infty)$ (Carleson embedding with loss). We recall that the case of the unit ball of $\mathbb{C}^n$ was treated in \cite{sehba4}. The study of Carleson embeddings started with L. Carleson in the setting of Hardy spaces (see \cite{Carleson1,Carleson2}) and was then extended to various analytic function spaces; for classical Bergman spaces of various domains, contributions can be found for example in \cite{CW,Hastings, Luecking1,Luecking2,Luecking3,NS}.
\medskip

Let  $\mu$ be a positive measure on $\mathbb{C}_{+}$. The Berezin transform $\widetilde{\mu}$ of  $\mu$ is the function defined by
$$  \widetilde{\mu}(z)= \int_{\mathbb{C}_{+}}\frac{\mathrm{Im}( z)^{2+\alpha}}{| \omega-\overline{z}|^{2(2+\alpha)}}d\mu(\omega), ~~ \forall~z \in \mathbb{C}_{+}.
   $$

Let $(X, \|.\|_{X})$ and $(Y, \|.\|_{Y})$ be two quasi-normed vector spaces. We say that $X$ is continuously embedded into $Y$, if there exists a constant  $C>0$ such that for all $f \in X$, 
\begin{equation}\label{eq:inaaaqay}
\|f\|_{Y} \leq C\|f\|_{X}.
 \end{equation}

The following is our result on Carleson embedding with loss for Bergman-Orlicz spaces.
\begin{theorem}\label{pro:mainqaaqaq6}
Let $\alpha>-1$,  $\Phi_{1} \in  \mathscr{U}\cap \nabla_{2}$ and $\Phi_{2} \in \mathscr{L} \cup \mathscr{U}$ and,    $\mu$ a positive measure on $\mathbb{C}_{+}$. Suppose that $t\mapsto \frac{\Phi_{1}(t)}{\Phi_{2}(t)}$ is non-decreasing on $(0, \infty)$ and
\begin{equation}\label{eq:inaaaqaay}
\int_{0}^{t}\frac{\Phi_{1}\circ\Phi_{2}^{-1}(s)}{s^{2}}ds\leq C\frac{\Phi_{1}\circ\Phi_{2}^{-1}(t)}{t}, ~~\forall ~t>0, 
 \end{equation}
where $C$ is a constant that depends only on $\Phi_{1}$ and $\Phi_{2}$. The following assertions are equivalent.
\begin{itemize}
\item[i)]  The Bergman$-$Orlicz space $A_{\alpha}^{\Phi_{1}}(\mathbb{C}_{+})$ continuously embeds into  $L^{\Phi_{2}}(\mathbb{C}_{+}, d\mu)$. 
 \item[ii)] The Berezin transform $\widetilde{\mu}$ belongs to $L^{\Phi_{3}}(\mathbb{C}_{+}, dV_{\alpha})$, where $\Phi_{3}$ is the complementary function of $\Phi_{1}\circ\Phi_{2}^{-1}$.
\end{itemize}
\end{theorem}

\begin{remark}
Condition (\ref{eq:inaaaqaay})  insures that the growth function $\Phi_3$ belongs to $\mathscr{U}$ and it is $\nabla_2$. 
\end{remark}
The proof of Theorem~\ref{pro:mainqaaqaq6} combines, in the direction ii)$\Rightarrow$i) in Theorem \ref{pro:mainqaaq6}, the local estimate for $\Phi_2(|F|)$, the comparison between averaging functions and the Berezin transform, the atomic decomposition together with Khintchine's inequality.
Let $\phi$ be as above and $\beta >-1$, we define the measure $\mu_{\phi,\beta}$  by
\begin{equation}\label{eq:inpay}
\mu_{\phi,\beta}(E):= |\phi^{-1}(E)|_{\beta}, 
\end{equation}
for any Borel set $E$ contained in $\mathbb{C}_{+}$. 
\medskip

Let $(X, \|.\|_{X})$ and $(Y, \|.\|_{Y})$ be two quasi-normed vector spaces. We say that a operator $T$ is bounded from $X$ to $Y$ if there exists a constant $C>0$ such that for all $f \in X$, 
\begin{equation}\label{eq:fo4naqAQ1l}
\|T(f)\|_{X}\leq C\|f\|_{Y}.
\end{equation}

The following follows directly from Theorem \ref{pro:mainqaaqaq6}. For classical Bergman spaces, see for example \cite{Ueki}.
\begin{corollary}\label{pro:mainqaq6}
Let $\alpha, \beta>-1$,   $\Phi_{1} \in  \mathscr{U}\cap \nabla_{2}$ and $\Phi_{2} \in \mathscr{L} \cup \mathscr{U}$ and, let $\phi : \mathbb{C}_{+}\longrightarrow \mathbb{C}_{+}$ be a holomorphic function. Suppose that $t\mapsto \frac{\Phi_{1}(t)}{\Phi_{2}(t)}$ is non-decreasing on $(0, \infty)$ and Relation (\ref{eq:inaaaqaay})  is satisfied. The following assertions are equivalent.
\begin{itemize}
\item[i)] The composition operator $C_{\phi}: A_{\alpha}^{\Phi_{1}}(\mathbb{C}_{+})\longrightarrow A_{\beta}^{\Phi_{2}}(\mathbb{C}_{+}) $ is bounded;
\item[ii)] The Berezin transform $\widetilde{\mu_{\phi,\beta}}$ belongs to $L^{\Phi_{3}}(\mathbb{C}_{+}, dV_{\alpha})$, where $\Phi_{3}$ is  the complementary function of $\Phi_{1}\circ\Phi_{2}^{-1}$.
\end{itemize}
\end{corollary} 

We propose the following abbreviation $\mathrm{ A}\lesssim \mathrm{ B}$ for the inequalities $\mathrm{ A}\leq C\mathrm{ B}$, where $C$ is a positive constant independent of the main parameters. If $\mathrm{ A}\lesssim \mathrm{ B}$ and $\mathrm{ B}\lesssim \mathrm{ A}$, then we write $\mathrm{ A}\sim \mathrm{ B}$.
\medskip

In the next section, we present several useful results needed for the proofs our results. The atomic decomposition of our spaces in consideration is proved in Section 3 as well as Theorem \ref{pro:mainqaaqaq6}. We close the paper with a conclusion as last section.

In all that follows, the letter $C$ will be used for nonnegative constants independent of the relevant variables that may change from one occurrence to another. Constants with subscript, such as $C_{s}$, may also change in different occurrences, but depend on the parameters mentioned in it. 

\section{Preliminaries.}

In this section, we present some useful results needed in our presentation.

\subsection{Some properties of growth functions.}

Let  $\Phi \in \mathscr{C}^{1}([0, \infty))$  a growth function. The lower and the upper indices of $\Phi$ are respectively defined by
$$ a_\Phi:=\inf_{t>0}\frac{t\Phi'(t)}{\Phi(t)}
  \hspace*{1cm}\textrm{and} \hspace*{1cm} b_\Phi:=\sup_{t>0}\frac{t\Phi'(t)}{\Phi(t)}.          $$
  We refer to \cite[Lemma 3.1]{djesehb} and \cite[Lemma 2.1]{sehbaedgc} for the following.
  \begin{lemma}
Let $\Phi \in \mathscr{C}^{1}([0, \infty))$ be a growth function. The following assertions are satisfied.
\begin{itemize}
\item[(i)] If  $\Phi \in \mathscr{L} \cup \mathscr{U}$, then  $0< a_\Phi\leq b_\Phi <\infty$.
\item[(ii)]  If  $\Phi \in \mathscr{U}$, then  $\Phi \in \nabla_{2}$ if and only if there is a constant  $C:=C_{\Phi}>0$ such that for all $t > 0$,
\begin{equation}\label{eq:saaui8n}
\int_{0}^{t}\frac{\Phi(s)}{s^{2}}ds \leq C \frac{\Phi(t)}{t}.
\end{equation}
\item[(iii)] If  $0< a_\Phi\leq b_\Phi <\infty$, then the function  $t\mapsto \frac{\Phi(t)}{t^{a_\Phi}}$ is increasing on $(0,\infty)$ while the function  $t\mapsto \frac{\Phi(t)}{t^{b_\Phi}}$ is decreasing on $(0,\infty)$.
\end{itemize}
\end{lemma}
We refer to \cite[Corollary 3.3]{djesehafeuto} for the following.
\begin{proposition}\label{pro:main3aaqq7}
Let $s \geq 1$  and   $\Phi \in \mathscr{C}^{1}([0, \infty))$ a growth function such that  $0< a_\Phi \leq b_\Phi < \infty $. Put, for  $t\geq 0$ 
$$  \Phi_{s}(t)=\Phi\left(t^{s/a_\Phi}\right).    $$
The following assertions are satisfied.
\begin{itemize}
\item[(i)] If $s=1$, then $\Phi_{s} \in \mathscr{U}$.
\item[(ii)] If $s>1$, then  $\Phi_{s} \in \mathscr{U} \cap \nabla_{2}$.
\end{itemize}
\end{proposition}

\begin{lemma}[Lemma 2.3, \cite{Tanohdjefeuto}]\label{pro:maiaqaq1q8}
Let  $\Phi$ be a growth function of lower type $p\in (0, \infty)$. The following assertions are satisfied: 
\begin{itemize}
\item[(i)]  $\Phi$ is equivalent to a continuous and increasing  growth function of lower type $p$.
\item[(ii)] The growth function $\Phi_{p}$ defined by
\begin{equation}\label{eq:suiaqaq8n}
\Phi_{p}(t)=\Phi\left( t^{1/p}\right),~~\forall~t\geq 0\end{equation}
is equivalent to a continuous, increasing and convex growth function.
\end{itemize}
\end{lemma}

\begin{remark}\label{pro:main 5aqaqq2pl}
In the following, we will assume that any growth function 
$\Phi$ of lower type $p \in (0, \infty)$ is continuous and increasing. Furthermore, the function  $\Phi_{p}$, defined in relation (\ref{eq:suiaqaq8n}), is a continuous and convex  growth function, thanks to Lemma \ref{pro:maiaqaq1q8}.
\end{remark}

\subsection{Bergman$-$Orlicz spaces on $\mathbb{C}_{+}$.}

Let $\alpha>-1$ and  $0<s< 1$. For a measurable subset  $E$ of $\mathbb{C}_{+}$ and  $z \in \mathbb{C}_{+}$, we will use the following notations
\begin{equation}\label{eq:inpaqmgay}
|E|_{\alpha}:=\int_{E}dV_{\alpha}(z),
 \qquad dV_\alpha(z):=y^\alpha\,dx\,dy
 \end{equation}
and 
\begin{equation}\label{eq:inpmgay}
\mathcal{D}_{s}(z):= \mathcal{D}(z,s\mathrm{Im}( z)),  
 \end{equation}
where $\mathcal{D}(a,r)$ is the disk centered at $a\in \mathbb{C}$  with radius $r>0$. We observe that for $z\in \mathbb{C}_{+}$, we have $\mathcal{D}_{s}(z)\subset \mathbb{C}_{+}$.

\medskip

Let $0<s < 1$ and  $\alpha \in \mathbb{R}$ and,  $z, \omega\in \mathbb{C_{+}}$. The following assertions are satisfied.
\begin{itemize}
\item[(i)] For  $\xi \in \mathcal{D}_{s}(z)$, we have
\begin{equation}\label{eq:espeqs1}
  \left(1- s \right) \mathrm{Im}( z) < \mathrm{Im}( \xi) < \left(1+ s \right) \mathrm{Im}( z)\end{equation}
and 
\begin{equation}\label{eq:espeqs12}
  \left(2- s \right) \mathrm{Im}( z) < |\xi-\overline{z}| < \sqrt{2}\left(2+ s \right) \mathrm{Im}( z).\end{equation}
When  $\alpha  >-1$, we also have 
\begin{equation}\label{eq:eqas1}
 |\mathcal{D}_{s}(z)|_{\alpha}\sim \mathrm{Im}( z)^{2+\alpha}.  
 \end{equation} 
\item[(ii)] Let $0<s < -1+ \sqrt{2}$, and put  $s':=s(2+s)$.  If   $ |z-\omega| < s\mathrm{Im}( z)$  then 
 \begin{equation}\label{eq:iaqnaqea}
  \mathcal{D}_{s}(\omega) \subset \mathcal{D}_{s'}(z).
 \end{equation} 
\item[(iii)] Let $0<s < 1/2$ and put  $s'':=\frac{s}{1-s}$.  Then 
 \begin{equation}\label{eq:iaqnaqeapq}
  \chi_{\mathcal{D}_{s}(\omega)}(z) =\chi_{\mathcal{D}_{s''}(z)}(\omega).
 \end{equation} 
\end{itemize}

The following proposition provides local weighted estimates for analytic functions in $\mathbb{C}_{+}$.

\begin{proposition}\label{pro:main1q8}
Let $\alpha\in\mathbb{R}$, and let $\Phi$ be a growth
function of lower type $p\in(0,\infty)$.
Choose $s\in(0,1)$ such that $0<s<-1+\sqrt{2}$, and set
$s':=s(2+s)$. Let $F$ be analytic on $\mathbb{C}_{+}$, and let
$z,\omega\in\mathbb{C}_{+}$. Then the following
assertions hold.

\begin{itemize}
\item[(i)] If $|z-\omega|<s\operatorname{Im}(z)$, then
\begin{equation}\label{eq:inehay}
\Phi(|F(\omega)|)
\leq
\frac{C_{s,\alpha}}{\operatorname{Im}(z)^{2+\alpha}}
\iint_{\mathcal{D}_{s'}(z)}
\Phi(|F(u+iv)|)\,v^\alpha\,du\,dv,
\end{equation}
where
\begin{equation}\label{eq:inehay34}
\begin{aligned}
C_{s,\alpha}:=\frac{1}{\pi s^2}
\max\left\{    \left(1- s \right)^{-\alpha}; \left(1+ s \right)^{-\alpha}   \right\}
\times
\max\left\{    \left(1- s' \right)^{-(2+\alpha)}; \left(1+ s' \right)^{-(2+\alpha)}   \right\}.
\end{aligned}
\end{equation}

\item[(ii)] Let $n \geq 1$.  If  $|z-\omega|<\frac{s^{n}}{2}\operatorname{Im}(z)$, then
\begin{equation}\label{eq:inegqhay}
\Phi(|F(z)-F(\omega)|)
\leq C_{p}(1+s)^{2}s^{(n-1)p-2}
\iint_{\mathcal{D}_{s'}(z)}
\Phi(|F(u+iv)|)\,\frac{du\,dv}{v^2},
\end{equation}
where $C_p$ is a constant depending only on $p$.
\end{itemize}
\end{proposition}

\begin{proof}
Put $$
\Phi_p(t):=\Phi\bigl(t^{1/p}\bigr),\qquad t\geq 0.
$$
By Remark~\ref{pro:main 5aqaqq2pl}, we may assume that $\Phi$ is continuous and increasing. Moreover, $\Phi_p$ is a continuous and convex growth function.

\medskip
\noindent
\textbf{(i) Proof of \eqref{eq:inehay}.}

Since  $|F|^{p}$ is subharmonic on $\mathbb{C}_ {+}$ and  $\Phi_{p}$ is convex, we have
\begin{align*}
\Phi(|F(\omega)|)=\Phi_{p}(|F(\omega)|^{p}) &\leq  \Phi_{p}\left(  \frac{1}{\pi (s\mathrm{Im}( \omega))^{2}}\int \int_{\mathcal{D}_{s}(\omega)}|F(u+iv)|^{p}dudv  \right)  \\
&\leq \frac{1}{\pi s^{2}} \times \frac{1}{\mathrm{Im}( \omega)^{2}}\int \int_{\mathcal{D}_{s}(\omega)}\Phi_{p}(|F(u+iv)|^{p})dudv,
\end{align*}
thanks to Jensen's inequality.

Since  $0<s<-1+\sqrt{2}$ and
$|z-\omega|<s\operatorname{Im}(z)$, we have
$$
\mathcal{D}_{s}(\omega)\subset\mathcal{D}_{s'}(z).
$$
Moreover, for $u+iv \in \mathcal{D}_{s}(\omega)$ and $\omega \in \mathcal{D}_{s'}(z)$, the inequalities in
\eqref{eq:espeqs1} imply that
$$ \frac{1}{v^{\alpha}} \leq  \max\left\{    \left(1- s \right)^{-\alpha}; \left(1+ s \right)^{-\alpha}   \right\} \frac{1}{\mathrm{Im}( \omega)^{\alpha}} $$
and 
$$ \frac{1}{\mathrm{Im}( \omega)^{2+\alpha}} \leq  \max\left\{    \left(1- s' \right)^{-(2+\alpha)}; \left(1+ s' \right)^{-(2+\alpha)}   \right\} \frac{1}{\mathrm{Im}( z)^{2+\alpha}}. $$
Consequently,
\begin{align*}
\Phi(|F(\omega)|)
&\leq \frac{1}{\pi s^{2}} \times \frac{1}{\mathrm{Im}( \omega)^{2}}\int \int_{\mathcal{D}_{s}(\omega)}\Phi(|F(u+iv)|)\frac{v^{\alpha}}{v^{\alpha}}dudv, \\
&\leq \frac{C_{s,\alpha}}{\mathrm{Im}( z)^{2+\alpha}}\int \int_{\mathcal{D}_{s'}(z)}\Phi(|F(u+iv)|)v^{\alpha}dudv, 
\end{align*}
where $C_{s,\alpha}$ is the constant defined in \eqref{eq:inehay34}.

\medskip
\noindent
\textbf{(ii) Proof of \eqref{eq:inegqhay}.} 

Since  $F$ is an analytic function  on $\mathbb{C_{+}}$, we have $$   F(z)-F(\omega)=\int_{[z,\omega]} F'(\zeta)d\zeta, $$
where $[z,\omega]$ denotes the line segment joining $z$ and $\omega$.

Now, let  $\zeta\in[z,\omega]$ and  put $r:=\frac{s}{2}\operatorname{Im}(z)$. Since the closed disk
$\overline{D(\zeta,r)}$ is contained in $\mathbb{C}_+$, by Cauchy's integral formula, we obtain 
$$ |F'(\zeta)| \leq \frac{1}{2\pi r}\int_{0}^{2\pi} |F(\zeta +re^{it})| dt.     $$

For every $\zeta\in[z,\omega]$ and $t\in[0,2\pi]$,
we have
$$ |\zeta +re^{it}-z| 
\leq   r + |\zeta-z| \leq  r + |\omega-z|  < \frac{s^{n}}{2}\mathrm{Im}( z) + \frac{s}{2}\mathrm{Im}( z)   \leq s\mathrm{Im}( z).
   $$

Thus, applying \eqref{eq:inehay} with \(\alpha=-2\) and replacing \(\omega\) by \(\zeta+re^{it}\), we obtain
$$ |F(\zeta +re^{it})| \leq  \Phi^{-1}\left(C_{s,\alpha}\int \int_{\mathcal{D}_{s'}(z)}\Phi(|F(u+iv)|)\frac{dudv}{v^{2}}\right).      $$
Note that, when $\alpha=-2$, the constant $C_{s,\alpha}$ is given by
\[
C_{s}=C_{s,-2} :=\frac{1}{\pi}\left( 1+\frac{1}{s}   \right)^{2}.
\]

Using the preceding estimates and the fact that the length of \([z,\omega]\) is \(|z-\omega|\), we get
\begin{align*}
| F(z)-F(\omega)| &\leq  \int_{[z,\omega]} |F'(\zeta)|d\zeta \\
&\leq  \int_{[z,\omega]} \left( \frac{1}{2\pi r}\int_{0}^{2\pi} |F(\zeta +re^{it})| dt  \right)d\zeta \\
&\leq  \int_{[z,\omega]} \left( \frac{1}{2\pi r}\int_{0}^{2\pi} \left(  \Phi^{-1}\left(C_{s,\alpha}\int \int_{\mathcal{D}_{s'}(z)}\Phi(|F(u+iv)|)\frac{dudv}{v^{2}}\right)  \right) dt  \right)d\zeta \\
&= |z-\omega| \times \frac{1}{2\pi r} \times 2\pi \times \Phi^{-1}\left(C_{s,\alpha}\int \int_{\mathcal{D}_{s'}(z)}\Phi(|F(u+iv)|)\frac{dudv}{v^{2}}\right) \\
& \leq \frac{s^{n}}{2}\operatorname{Im}(z)\times \frac{1}{ \frac{s}{2}\operatorname{Im}(z)}\Phi^{-1}\left(C_{s,\alpha}\int \int_{\mathcal{D}_{s'}(z)}\Phi(|F(u+iv)|)\frac{dudv}{v^{2}}\right).
\end{align*}
We deduce that
$$  \Phi\left(   \frac{| F(z)-F(\omega)|}{s^{n-1}}  \right)  \leq C_{s}\int \int_{\mathcal{D}_{s'}(z)}\Phi(|F(u+iv)|)\frac{dudv}{v^{2}}.  $$
By the lower type property of $\Phi$ with exponent $p\in(0,\infty)$, we have
\begin{align*}
\Phi \left(| F(z)-F(\omega)|\right) &\leq C_{p}s^{(n-1)p} \Phi\left(   \frac{| F(z)-F(\omega)|}{s^{n-1}}  \right) \\
&\leq C_{s}C_{p}s^{(n-1)p}\int \int_{\mathcal{D}_{s'}(z)}\Phi(|F(u+iv)|)\frac{dudv}{v^{2}}.
\end{align*}
This completes the proof.
\end{proof}

The following result follows from inequality (\ref{eq:inehay}). Therefore, the proof will be omitted.

 \begin{proposition}\label{pro:main1aq1apaq8}
Let $\alpha>-1$ and  $\Phi$ a growth function  of lower type. There exists a constant $C:=C_{\alpha, \Phi}>0$ such that for 
$ F \in A^{\Phi}_{\alpha}(\mathbb{C_{+}})$, 
 \begin{equation}\label{eq:inegalehay}
 |F(z)|\leq C\Phi^{-1}\left(\frac{1}{ \mathrm{Im}( z)^{2+\alpha}}\right)\|F\|_{A^{\Phi}_{\alpha}}^{lux},~~\forall~z \in \mathbb{C_{+}}.
 \end{equation}
 \end{proposition}

Let us prove the following.
\begin{lemma}\label{pro:main12amq0}
Let $\alpha>-1$ and  $\Phi$ a growth function  of lower type. The Bergman-Orlicz space $(A^{\Phi}_{\alpha}(\mathbb{C_{+}}),  \|.\|_{A^{\Phi}_{\alpha}}^{lux})$ is a closed subspace of the Orlicz space  $(L^{\Phi}(\mathbb{C_{+}}, dV_{\alpha}),  \|.\|_{L^{\Phi}_{\alpha}}^{lux})$.
\end{lemma}

\begin{proof}
Let $(F_{n})_{n}$ be a sequence of elements of $A^{\Phi}_{\alpha}(\mathbb{C_{+ }})$ which converges to $F \in (L^{\Phi}(\mathbb{C_{+}}, dV_{\alpha}),  \|.\|_{L^{\Phi}_{\alpha}}^{lux})$.

\medskip

Let us show that $F$ belongs to $A^{\Phi}_{\alpha}(\mathbb{C_{+ }})$.

\medskip

Let $\varepsilon >0$. Since  $(F_{n})_{n}$ is a Cauchy sequence in  $A^{\Phi}_{\alpha}(\mathbb{C_{+ }})$, there exists $n_{\varepsilon}\in \mathbb{N}$ such that for all $ n,m \geq n_{\varepsilon}$, $\|F_{n}-F_{m}\|_{A ^{\Phi}_{\alpha}}^{lux}\leq \varepsilon$.
Let  $z_{0} \in \mathbb{C_{+ }}$ and $r>0$ such that  $\overline{\mathcal{D}(z_{0}, r)}:=\{z \in \mathbb{C_{+ }} :  |z-z_{0}| \leq r    \}$ is contained in $\mathbb{C_{+ }}$. Put  
$$ A_{r}(z_{0}):=\inf\left\{\mathrm{Im}(z) : z\in \overline{\mathcal{D}(z_{0}, r)}\right\}=\mathrm{Im}(z_{0})-r>0.$$
For $n,m \geq n_{\varepsilon}$ and $z \in \overline{\mathcal{D}(z_{0}, r)}$, we have
$$ |F_{n}(z)-F_{m}(z)| \lesssim \Phi^{-1}\left(\frac{1}{ \mathrm{Im}( z)^{2+\alpha}}\right)\|F_{n}-F_{m}\|_{A^{\Phi}_{\alpha}}^{lux}\lesssim \Phi^{-1}\left(\frac{1}{ ( A_{r}(z_{0}))^{2+\alpha}}\right)\times \varepsilon,   $$
thanks to Proposition \ref{pro:main1aq1apaq8}. We deduce that the sequence $(F_{n})_{n}$ converges uniformly on
$\overline{\mathcal{D}(z_{0}, r)}$. It follows that there exists $G$ an analytic function on $\mathbb{C_{+}}$, such that the sequence  $(F_{n})_{n}$ converges uniformly to $G$ on any compact subset of  $\mathbb{C_{+}}$. 
For $n \geq n_{\varepsilon}$, we have
$$ \int_{\mathbb{C_{+}}}\Phi\left(\frac{|F_{n}(z)-G(z)|}{\varepsilon}\right)dV_{\alpha}(z) \leq  \liminf_{m\to +\infty} \int_{\mathbb{C_{+}}}\Phi\left(\frac{|F_{n}(z)-F_{m}(z)|}{\varepsilon}\right)dV_{\alpha}(z) \leq 1,      $$
thanks to Fatou's lemma. Therefore, the sequence $(F_{n})_{n}$ converges to $G$ in $A^{\Phi}_{\alpha}(\mathbb{C_{+}}).$ As, 
 $$ \|F-G\|_{A^{\Phi}_{\alpha}}^{lux} = \|F-G\|_{L_{\alpha}^{\Phi}}^{lux} \lesssim \|F_{n_{\varepsilon}}-F\|_{L_{\alpha}^{\Phi}}^{lux}+\|F_{n_{\varepsilon}}-G\|_{L_{\alpha}^{\Phi}}^{lux}\lesssim \varepsilon.       $$
It follows that $F\in A^{\Phi}_{\alpha}(\mathbb{C_{+}})$.
\end{proof}

Fix $\beta \in \left\{0, \frac{1}{3}\right\}$. A $\beta-$dyadic interval is any interval $I$ of $\mathbb{R}$ of the form
 $$   2^{-j}(\lbrack 0,1) +k+(-1)^{j}\beta),    $$
 where $k,j\in \mathbb{Z}$. We denote by $\mathcal{D}_{j}^{\beta}$ the set of $\beta-$dyadic intervals $I$ such that $|I|=2^{-j}$.  Put
 $$ \mathcal{D}^{\beta}:=\bigcup_{j}\mathcal{D}_{j}^{\beta}.$$
The family \(\mathcal{D}^{\beta}\) forms a dyadic grid and enjoys the following standard properties (see, for instance, \cite{djesehaqb,PottReguera,sehba3}):
 \begin{itemize}
 \item[-] for all $I,J\in \mathcal{D}^{\beta}$, we have $I\cap J \in \left\{\emptyset ; I; J\right\}$,
 \item[-] for each fixed  $j\in \mathbb{Z}$, if $I \in \mathcal{D}_{j}^{\beta}$ then there exists a unique $J \in \mathcal{D}_{j-1}^{\beta}$ such that  $I\subset J$,
 \item[-] for each fixed $j\in \mathbb{Z}$, if $I \in \mathcal{D}_{j}^{\beta}$ then there exists $I_{1}, I_{2} \in \mathcal{D}_{j+1}^{\beta}$ such that
  $I=I_{1}\cup I_{2}$ and $I_{1}\cap I_{2}=\emptyset$,
  \item[-] for every interval \(I\subset\mathbb{R}\), there exist \(\beta\in\left\{0,\frac13\right\}\) and an interval \(J\in\mathcal{D}^{\beta}\) such that $I\subset J$ and $|J|\leq 6|I|$.
  \end{itemize}

\medskip

Let \(I\subset\mathbb{R}\) be an interval of finite, nonzero length. The Carleson square associated with \(I\) is defined by
\[
Q_I:=\left\{x+iy\in\mathbb{C}_{+}:x\in I,\quad 0<y<|I|\right\}.
\]
We also define the upper half of \(Q_I\) by
\[
T_I:=\left\{x+iy\in\mathbb{C}_{+}:x\in I,\quad
\frac{|I|}{2}<y<|I|\right\}.
\]
Clearly, \(T_I\subset Q_I\). Furthermore, for every \(\alpha>-1\), the weighted measures of \(Q_I\) and \(T_I\) are comparable. More precisely, there exist positive constants \(c_\alpha\) and \(C_\alpha\), depending only on \(\alpha\), such that
\begin{equation}\label{eq:carleson_comparability}
c_\alpha |Q_I|_\alpha
\le |T_I|_\alpha
\le C_\alpha |Q_I|_\alpha.
\end{equation}
Finally, for each $\beta\in\left\{0,\frac{1}{3}\right\}$, the family
  $\{T_{Q_{I}}\}_{I\in\mathcal{D}^{\beta}}$ consists of pairwise disjoint sets.

Now let \(\alpha>-1\) and let \(F\) be a measurable function on \(\mathbb{C}_{+}\). The Hardy–Littlewood maximal operator associated with Carleson squares is defined by
$$ \mathcal{M}_{\alpha}(F)(z):= \sup_{I \subset \mathbb{R}}\frac{\chi_{Q_{I}}(z)}{|Q_{I}|_{\alpha}}\int_{Q_{I}}| F(\omega)| dV_{\alpha}(\omega),~~ \forall~ z \in \mathbb{C}_{+},     $$ 
whenever the averages are well defined. Here, the supremum is taken over all finite-length intervals \(I\subset\mathbb{R}\). Similarly, for $\beta \in \left\{0; 1/3\right\}$, we define its dyadic version  $\mathcal{M}_{\alpha}^{\mathcal{D}^{\beta}}(F)$ by taking the supremum in the above definition this time only on the intervals contained in $\mathcal{D}^{\beta}$. 

We thus obtain the following pointwise estimate:
\begin{equation}\label{eq:ooua5x1l}
\mathcal{M}_{\alpha}^{\mathcal{D}^{\beta}}(F)(z)\le\mathcal{M}_{\alpha}(F)(z)
\le
6^{2+\alpha}
\sum_{\beta\in\{0,\frac13\}}
\mathcal{M}_{\alpha}^{\mathcal{D}^{\beta}}(F)(z),
\qquad z\in\mathbb{C}_{+}.
\end{equation}

In the following result, we recall a characterization of the boundedness of maximal operators on weighted Orlicz spaces.

\begin{proposition}[\cite{djesehafeuto}  ]\label{pro:mainplaqaqq6}
Let   $\alpha>-1$ and $\Phi \in \mathscr{U}$. The following assertions are equivalent.
\begin{itemize}
\item[(i)] $\Phi \in \nabla_{2}$.
\item[(ii)] $\mathcal{M}_{\alpha}: L^{\Phi}(\mathbb{C}_{+}, dV_{\alpha})\longrightarrow L^{\Phi}(\mathbb{C}_{+}, dV_{\alpha})$ is bounded.
\item[(iii)] For $\beta\in\{0,\frac13\}$,  $\mathcal{M}_{\alpha}^{\mathcal{D}^{\beta}}: L^{\Phi}(\mathbb{C}_{+}, dV_{\alpha})\longrightarrow L^{\Phi}(\mathbb{C}_{+}, dV_{\alpha})$ is bounded.
\end{itemize}
\end{proposition}

Let  $\Phi$ be a growth function  of lower type. For $f\in L^{\Phi}(X, d\mu)$, put 
$$ \|f\|_{L_{\mu}^{\Phi}}:= \int_{X}\Phi\left(| F(x)|\right)d\mu(x).   $$
If  $\Phi \in \mathscr{C}^{1}(\mathbb{R}_{+})$ is a growth function such that   $0< a_\Phi\leq b_\Phi <\infty$, then we have the following inequalities:
\begin{equation}\label{eq:inegehay}
\|f\|_{L_{\mu}^{\Phi}} \lesssim \max\left\{ \left(  \|f\|_{L_{\mu}^{\Phi}}^{lux} \right)^{a_\Phi}; \left( \|f\|_{L_{\mu}^{\Phi}}^{lux}\right)^{b_\Phi}\right\}
\end{equation}
and
\begin{equation}\label{eq:inmehay}
\|f\|_{L_{\mu}^{\Phi}}^{lux} \lesssim \max\left\{ \left( \|f\|_{L_{\mu}^{\Phi}}\right)^{1/a_\Phi} ;  \left( \|f\|_{L_{\mu}^{\Phi}}  \right)^{1/b_\Phi}\right\}.
\end{equation}

Wee will simply use the notation  $L^{\Phi}(\mathbb{R}):= L^{\Phi}(\mathbb{R}, dx)$, where $dx$ is the Lebesgue measure on $\mathbb{R}$ and 
$\|F\|_{A^{\Phi}_{\alpha}}:= \|F\|_{L_{\alpha}^{\Phi}},$
when $F \in A^{\Phi}_{\alpha}(\mathbb{C_{+}})$.

\medskip

Let  $\Phi$ a be growth function  of lower type. 
The Hardy-Orlicz space on  $\mathbb{C_{+}}$,  $H^{\Phi}(\mathbb{C_{+}})$ consists of all analytic functions $F$ on $\mathbb{C_{+}}$ that satisfy 
$$   \|F\|_{H^{\Phi}}^{lux}:=\sup_{y> 0}\|F(.+iy)\|_{L^{\Phi}}^{lux} < \infty.   $$
If $\Phi$ is  convex  then  $\left(H^{\Phi}(\mathbb{C_{+}}), \|.\|_{H^{\Phi}}^{lux}\right)$ is a Banach space (see \cite{djesehb, Jan, Jan1}).

We can find the following result in \cite{djesehafeuto}.

\begin{proposition}\label{pro:main6apmqaalqpqq4}
Let  $\Phi \in \mathscr{L} \cup \mathscr{U}$ and  $ F\in H^{\Phi}(\mathbb{C_{+}})$.  The following assertions are satisfied:  
\begin{itemize}
\item[(i)] The functions $y\mapsto \|F(.+iy)\|_{L^{\Phi}}$ and  $y\mapsto \|F(.+iy)\|_{L^{\Phi}}^{lux}$ are non-increasing on $(0, \infty)$.
\item[(ii)] There exists a unique function 
     $f\in  L^{\Phi}\left(\mathbb{R}\right)$ such that  
  $  f(x)=\lim_{y\to 0}F(x+iy),$
 for almost all $x\in \mathbb{R}$.
 \end{itemize}
 Moreover,
  \begin{equation}\label{eq:sqapq8n}
  \|F\|_{H^{\Phi}}^{lux}=\lim_{y \to 0}\|F(.+iy)\|_{L^{\Phi}}^{lux}=\|f\|_{L^{\Phi}}^{lux}.
  \end{equation}
\end{proposition}

\begin{lemma}\label{pro:main11aaq8}
Let $\alpha>-1$,   $\Phi \in \mathscr{L} \cup \mathscr{U}$ and $ F \in A^{\Phi}_{\alpha}(\mathbb{C_{+}})$.  For $ \varepsilon>0$, put 
\begin{equation}\label{eq:inegalaqehay}
 F_{\varepsilon}(z)=  F(z+i\varepsilon), ~~\forall~x+iy \in \mathbb{C_{+}}.
\end{equation}
The following assertions are satisfied.
\begin{itemize}
\item[(i)] The functions  $y\mapsto \|F(.+iy)\|_{L^{\Phi}}$ and  $y\mapsto \|F(.+iy)\|_{L^{\Phi}}^{lux}$ are non-increasing on $(0,\infty)$;
\item[(ii)] The function $ F_{\varepsilon}$ belongs to $  H^{\Phi}(\mathbb{C}_{+}) \cap A^{\Phi}_{\alpha}(\mathbb{C_{+}})$.  
\end{itemize}
Moreover, 
   \begin{equation}\label{eq:indehay}
   \sup_{\varepsilon>0}\|F_{\varepsilon}\|_{A^{\Phi}_{\alpha}}^{lux}= \|F\|_{A^{\Phi}_{\alpha}}^{lux} \hspace*{0.5cm}\textrm{and} \hspace*{0.5 cm} \lim_{\varepsilon \to 0}\|F_{\varepsilon}-F\|_{A^{\Phi}_{\alpha}}^{lux}=0.
   \end{equation}
\end{lemma}

\begin{proof}
 Assume that $F\not\equiv 0$ because there is nothing to show when $F\equiv 0$. Without loss of generality, we will also assume that $\|F\|_{A^{\Phi}_{\alpha}}^{lux}=1$.
 
 \medskip
 
$i)$ Let $0<s < 1$ such that $0<s<-1+\sqrt{2}$, and set
$s':=s(2+s)$. For  $y>0$ and $x \in \mathbb{R}$, using Proposition \ref{pro:main1q8}, we have
\begin{align*}
\Phi\left(| F_{\varepsilon}(x+iy)|\right) &=  \Phi\left(| F(x+i(y+\varepsilon))|\right) \\
&\lesssim \frac{1}{( y+\varepsilon)^{2+\alpha}}\int \int_{\mathcal{D}_{s'}(x+i(y+\varepsilon))}\Phi\left(| F(u+iv)|\right)dV_{\alpha}(u+iv) \\
&\lesssim \frac{1}{( y+\varepsilon)^{2+\alpha}}\int_{|x-u|\leq y+\varepsilon} \int_{|y+\varepsilon-v|\leq y+\varepsilon}\Phi\left(| F(u+iv)|\right)v^{\alpha}dudv. 
\end{align*}
Hence using Fubini's theorem, we obtain
\begin{align*}
\int_{-\infty}^{+\infty}\Phi\left(| F_{\varepsilon}(x+iy)|\right)dx 
&\lesssim \frac{1}{( y+\varepsilon)^{2+\alpha}}\int_{-\infty}^{+\infty}\int_{ -y-\varepsilon+u}^{y+\varepsilon+u} \int_{0}^{\infty}\Phi\left(| F(u+iv)|\right)v^{\alpha}dudvdx \\
&= \frac{1}{( y+\varepsilon)^{2+\alpha}}\int_{-\infty}^{+\infty}\left(\int_{ -y-\varepsilon+u}^{y+\varepsilon+u}dx \right) \int_{0}^{\infty}\Phi\left(| F(u+iv)|\right)v^{\alpha}dudv \\
&\lesssim \frac{1}{\varepsilon^{1+\alpha}}\int_{0}^{\infty}\int_{-\infty}^{+\infty} \Phi\left(| F(u+iv)|\right)v^{\alpha}dudv.
\end{align*}
We deduce that 
$$  \sup_{y>0} \int_{\mathbb{R}}\Phi\left(| F_{\varepsilon}(x+iy)|\right)dx < \infty.
  $$
Therefore, $F_{\varepsilon} \in H^{\Phi}(\mathbb{C}_{+})$. We know from Proposition \ref{pro:main6apmqaalqpqq4} that the functions $y\mapsto \|F_{\varepsilon}(.+iy)\|_{L^{\Phi}}$ and  $y\mapsto \|F_{\varepsilon}(.+iy)\|_{L^{\Phi}}^{lux}$ are non-increasing on $(0, \infty)$. Moreover,  there exists a unique function 
     $f_{\varepsilon}\in  L^{\Phi}\left(\mathbb{R}\right)$ such that  
  $$  f_{\varepsilon}(x):=\lim_{y\to 0}F_{\varepsilon}(x+iy)= F(x+i\varepsilon) ,   $$
 for almost all $x\in \mathbb{R}$. It follows that, for all $\varepsilon>0$ and $y>0$, 
\begin{equation}\label{eq:inegaqthay}
 \|F(.+i(y+\varepsilon))\|_{L^{\Phi}}^{lux} \leq  \|F(.+i\varepsilon)\|_{L^{\Phi}}^{lux}.
\end{equation}
For $0<y_{1}<y_{2}$, we can find $\varepsilon >0$ such that $y_{2}=  y_{1}+\varepsilon$. From (\ref{eq:inegaqthay}), we deduce that
 $$  \|F(.+iy_{2})\|_{L^{\Phi}}^{lux} \leq  \|F(.+iy_{1})\|_{L^{\Phi}}^{lux}.    $$
Likewise, we obtain that
$$  \|F(.+iy_{2})\|_{L^{\Phi}} \leq  \|F(.+iy_{1})\|_{L^{\Phi}}.    $$

\medskip

$ii)$ For $ \varepsilon>0$, we deduce that  $F_{\varepsilon} \in A^{\Phi}_{\alpha}(\mathbb{C_{+}})$ and   
\begin{equation}\label{eq:inegthay}
 \|F_{\varepsilon}\|_{A^{\Phi}_{\alpha}}^{lux}\leq \|F\|_{A^{\Phi}_{\alpha}}^{lux}.
 \end{equation}
Indeed, since the function $y\mapsto \|F(.+iy)\|_{L^{\Phi}}$ is non-increasing on $(0, \infty)$, we have
$$  \int_{0}^{\infty} \left\|F_{\varepsilon}(.+iy)\right\|_{L^{\Phi}} y^{\alpha}dy =\int_{0}^{\infty} \left\|F(.+i(y+\varepsilon))\right\|_{L^{\Phi}} y^{\alpha}dy \leq \int_{0}^{\infty} \left\|F(.+iy)\right\|_{L^{\Phi}} y^{\alpha}dy \leq 1. $$
As
$$ 
  F(z) = \lim_{\varepsilon\to 0} F_{\varepsilon}(z),   $$
for almost all $z \in \mathbb{C}_{+}$, using Fatou's lemma, we deduce that 
\begin{equation}\label{eq:inaldehay}
 \|F\|_{A^{\Phi}_{\alpha}}^{lux} \leq \sup_{\varepsilon>0}\|F_{\varepsilon}\|_{A^{\Phi}_{\alpha}}^{lux}. 
   \end{equation}
Indeed, 
$$ \int_{\mathbb{C}_{+}} \Phi\left(|F(z)|\right)dV_{\alpha}(z)  \leq \liminf_{\varepsilon \to 0}\int_{\mathbb{C}_{+}} \Phi\left(|F_{\varepsilon}(z)|\right)dV_{\alpha}(z) \leq \sup_{\varepsilon> 0}\int_{\mathbb{C}_{+}} \Phi\left(|F_{\varepsilon}(z)|\right)dV_{\alpha}(z).  $$
From (\ref{eq:inegthay}) and (\ref{eq:inaldehay}), we deduce that 
\begin{equation}\label{eq:inaldepmhay}
   \sup_{\varepsilon>0}\|F_{\varepsilon}\|_{A^{\Phi}_{\alpha}}^{lux}=\|F\|_{A^{\Phi}_{\alpha}}^{lux}. 
   \end{equation}
Since $\Phi \in \mathscr{L} \cup \mathscr{U}$, we deduce that  $0< a_\Phi \leq b_\Phi < \infty $.  
Put for  $t\geq 0$, 
$$  \widetilde{\Phi}(t)=\Phi\left(t^{2/a_\Phi}\right).    $$
By construction,  $\widetilde{\Phi} \in \mathscr{U} \cap \nabla_{2}$, according to Proposition \ref{pro:main3aaqq7}. For $\varepsilon>0$ and $z \in \mathbb{C}_{+}$, we have
$$ |F(z)-F_{\varepsilon}(z)| \lesssim \left(\mathcal{M}_{\alpha}\left(|F|^{a_\Phi/2}\right)(z) \right)^{2/a_\Phi}, $$
(see Proposition 3.26 in \cite{djesehafeuto}). 
Since $F\in L^{\Phi}(\mathbb{C}_{+}, dV_{\alpha})$, we deduce that $|F|^{a_\Phi/2}\in L^{\widetilde{\Phi}}(\mathbb{C}_{+}, dV_{\alpha})$. It follows that $\mathcal{M}_{\alpha}\left(|F|^{a_\Phi/2}\right) \in L^{\widetilde{\Phi}}(\mathbb{C}_{+}, dV_{\alpha})$, thanks to Proposition \ref{pro:mainplaqaqq6}. Using  the dominated convergence theorem, we have for all  $\lambda>0$,
$$ \lim_{\varepsilon \to 0} \int_{\mathbb{C}_{+}} \Phi\left(\lambda|F(z)-F_{\varepsilon}(z)|\right)dV_{\alpha}(z) = \int_{\mathbb{C}_{+}}\lim_{\varepsilon \to 0} \Phi\left(\lambda|F(z)-F_{\varepsilon}(z)|\right)dV_{\alpha}(z)=0. $$
The proof is complete.
\end{proof}

Denote by $B$  the usual beta function defined by
 $$  B(m,n) =\int_{0}^{\infty}\dfrac{u^{m-1}}{(1+u)^{m+n}}du , ~~\forall~ m, n > 0.   $$
 
 The following results can be found for example in \cite{bansahsehba}.
  
 \begin{lemma}\label{pro:main26}
  Let  $y > 0$ and $\alpha \in \mathbb{R}$. The integral 
 $$ \textit{J}_{\alpha}(y) =\int_{\mathbb{R}}\dfrac{dx}{|x+iy|^{\alpha}},    $$ 
 converges if and only if $\alpha > 1$. In this case,
 $$ \textit{J}_{\alpha}(y)=B\left(\frac{1}{2}, \frac{\alpha-1}{2} \right)y^{1-\alpha}.    $$
\end{lemma}

 \begin{lemma}\label{pro:main27}
Let $\alpha, \beta \in \mathbb{R}$ and $t>0$. The integral 	
\begin{equation}\label{eq:alfdyafintgama}
  \textit{I}(t) =\int_{0}^{\infty}\dfrac{y^{\alpha}}{(t+y)^{\beta}}dy, \end{equation}		
converges if and only if $\alpha >- 1$ and  $\beta-\alpha > 1$. In this case,
\begin{equation}\label{eq:alfdybeitete}
\textit{I}(t)=B(1+\alpha, \beta-\alpha-1)t^{-\beta+\alpha+1}.\end{equation}	
  \end{lemma}

Let us prove the following.
\begin{proposition}\label{pro:main18}
Let $\alpha>-1$, and   $\Phi$ a growth function of lower type $p \in (0, \infty)$.  For $m\geq 0$, $\varepsilon> 0$ and $z \in \mathbb{C_{+}}$, put
\begin{equation}\label{eq:inegalitedaqay}
  G_{\varepsilon,m}(z)=  \frac{1}{(-i\varepsilon z+1)^{m}}.
 \end{equation}
The following assertions are satisfied.
\begin{itemize}
\item[(i)] If  $mp>1$,  then  $G_{\varepsilon,m}\in H^{\Phi}(\mathbb{C}_{+})$,
\item[(ii)] If $mp>(\alpha +2)$, then  $G_{\varepsilon,m}\in A_{\alpha}^{\Phi}(\mathbb{C}_{+})$.  
\end{itemize}
 \end{proposition}
  
\begin{proof}
$(i)$ For  $z \in \mathbb{C_{+}}$, we have 
$$   G_{\varepsilon,m}(z)=  \frac{1}{(-i\varepsilon z+1)^{m}}= \frac{1}{(-i)^{m}(\varepsilon z+i)^{m}}.   $$
Since the function  $t\mapsto \frac{\Phi(t)}{t^{p}}$ is non-decreasing on $(0, \infty)$, we have using Lemma \ref{pro:main26}, that for $y>0$,
\begin{align*}
\int_{\mathbb{R}}\Phi(|G_{\varepsilon,m}(x+iy)|)dx &\lesssim \int_{\mathbb{R}}\frac{dx}{|\varepsilon x+i(1+\varepsilon y)|^{mp}} \\
&\lesssim B\left(\frac{1}{2}, \frac{mp-1}{2} \right)\frac{1/\varepsilon}{(1+\varepsilon y)^{mp-1}} \\
&\lesssim B\left(\frac{1}{2}, \frac{mp-1}{2} \right)\times  \varepsilon^{-1}.
\end{align*}
We conclude that,  $G_{\varepsilon,m}\in H^{\Phi}(\mathbb{C}_{+})$ for $mp>1$.

\medskip

$(ii)$ If $mp>\alpha +2$, then  $G_{\varepsilon,m}\in A_{\alpha}^{\Phi}(\mathbb{C}_{+})$. Indeed, we have
\begin{align*}
\int_{0}^{\infty}\int_{\mathbb{R}}\Phi(|G_{\varepsilon,m}(x+iy)|)y^{\alpha}dx dy  &\lesssim  \int_{0}^{\infty}\left(  \int_{\mathbb{R}}\frac{dx}{|\varepsilon x+i(1+\varepsilon y)|^{mp}} \right) y^{\alpha}dy \\
&=\left(1/\varepsilon\right)^{mp} B\left(\frac{1}{2}, \frac{mp-1}{2} \right)\int_{0}^{\infty}\frac{ y^{\alpha}}{(1/\varepsilon+y)^{mp-1}}dy\\
&=B\left(\frac{1}{2}, \frac{mp-1}{2} \right)B\left(\alpha +1, mp-\alpha-2 \right)\varepsilon^{-2-\alpha}, 
\end{align*}
thanks to Lemma \ref{pro:main26} and Lemma \ref{pro:main27}.
\end{proof}

\begin{lemma}\label{pro:main11aqaaqm8}
Let $\alpha, \beta>-1$ and  $\Phi_{1},\Phi_{2} \in  \mathscr{L} \cup \mathscr{U}$. $A^{\Phi_{1}}_{\alpha}(\mathbb{C_{+}})\cap A^{\Phi_{2}}_{\beta}(\mathbb{C_{+}})$ is a dense subspace of $A^{\Phi_{1}}_{\alpha}(\mathbb{C_{+}})$.
\end{lemma}

\begin{proof}
Since  $\Phi_{j} \in \mathscr{L} \cup \mathscr{U}$, we deduce that  $0< a_{\Phi_{j}}\leq b_{\Phi_{j}} <\infty$.

\medskip

Let  $F\in A^{\Phi_{1}}_{\alpha}(\mathbb{C_{+}})$. Let $m \in \mathbb{R}$ be such that  $ma_{\Phi_{2}} > \beta+2$. For  $\varepsilon>0$ and $z \in \mathbb{C_{+}}$, put 
\begin{equation}\label{eq:inaliedehay}
  F^{(\varepsilon)}(z) = G_{\varepsilon, m}(z)F_{\varepsilon}(z),   
\end{equation}
where $F_{\varepsilon}$ and $G_{\varepsilon, m}$ are the functions defined respectively  in (\ref{eq:inegalaqehay})
 and in (\ref{eq:inegalitedaqay}). 
 
\medskip
 
Let us show that $F^{(\varepsilon)}$ belongs to $A^{\Phi_{1}}_{\alpha}(\mathbb{C_{+}})$. We know that  $ F_{\varepsilon} \in  A^{\Phi_{1}}_{\alpha}(\mathbb{C_{+}})$ and $\|F_{\varepsilon}\|_{A^{\Phi_{1}}_{\alpha}}^{lux} \leq \|F\|_{A^{\Phi_{1}}_{\alpha}}^{lux}$, thanks to Lemma \ref{pro:main11aaq8}. As  $|G_{\varepsilon,m}(z)|<1$, for all  $z \in \mathbb{C_{+}}$, we have
$$ \int_{\mathbb{C}_{+}}\Phi_{1}\left( \frac{|F^{(\varepsilon)}(z)|}{\|F\|_{A^{\Phi_{1}}_{\alpha}}^{lux}}  \right)dV_{\alpha}(z) \leq \int_{\mathbb{C}_{+}}\Phi_{1}\left( \frac{|F_{\varepsilon}(z)|}{\|F\|_{A^{\Phi_{1}}_{\alpha}}^{lux}}  \right)dV_{\alpha}(z) \leq 1.     $$
We deduce that
 $F^{(\varepsilon)} \in A^{\Phi_{1}}_{\alpha}(\mathbb{C_{+}})$.

\medskip
 
We still have to prove that $F^{(\varepsilon)} \in A^{\Phi_{2}}_{\beta}(\mathbb{C_{+}})$.  For $z=x+iy \in \mathbb{C_{+}}$, we have 
$$ |F_{\varepsilon}(x+iy)| = |F(x+i(y+\varepsilon))| \lesssim \Phi_{1}^{-1}\left(\frac{1}{ (y+\varepsilon)^{2+\alpha}}\right)\|F\|_{A^{\Phi_{1}}_{\alpha}}^{lux} \lesssim \Phi_{1}^{-1}\left(\frac{1}{ \varepsilon^{2+\alpha}}\right)\|F\|_{A^{\Phi_{1}}_{\alpha}}^{lux},   $$
thanks to Proposition \ref{pro:main1aq1apaq8}. We deduce that 
$$ \Phi_{2}(|F_{\varepsilon}(z)G_{m,\varepsilon}(z)|) \lesssim \Phi_{2}( |G_{m,\varepsilon}(z)|).   $$
Since  $G_{m,\varepsilon} \in A^{\Phi_{2}}_{\beta}(\mathbb{C_{+}})$, thanks to Proposition \ref{pro:main18}, it follows that  $F^{(\varepsilon)} \in A^{\Phi_{2}}_{\beta}(\mathbb{C_{+}})$. 

\medskip

Now, for  $t\geq 0$, put 
$$  \widetilde{\Phi_1}(t)=\Phi_1\left(t^{2/a_{\Phi_{1}}}\right).    $$
By construction,  $\widetilde{\Phi_1} \in \mathscr{U} \cap \nabla_{2}$. Since $F\in L^{\Phi_{1}}(\mathbb{C}_{+}, dV_{\alpha})$, we deduce that $|F|^{a_{\Phi_{1}}/2}\in L^{\widetilde{\Phi_1}}(\mathbb{C}_{+}, dV_{\alpha})$. It follows that $\mathcal{M}_{\alpha}\left(|F|^{a_{\Phi_{1}}/2}\right) \in L^{\widetilde{\Phi_1}}(\mathbb{C}_{+}, dV_{\alpha})$, thanks to Proposition \ref{pro:mainplaqaqq6}. As
$$   \lim_{\varepsilon \to 0} F^{(\varepsilon)}(z) = F(z),   
   $$
for all  $z \in \mathbb{C_{+}}$ and
$$ 
 |F(z)-F_{\varepsilon}(z)| \lesssim \left(\mathcal{M}_{\alpha}\left(|F|^{a_{\Phi_{1}}/2}\right)(z) \right)^{2/a_{\Phi_{1}}},    $$
according to the dominated convergence theorem, we have 
 for all  $\lambda>0$,
$$ \lim_{\varepsilon \to 0} \int_{\mathbb{C}_{+}} \Phi_{1}\left(\lambda|F(z)-F^{(\varepsilon)}(z)|\right)dV_{\alpha}(z) = \int_{\mathbb{C}_{+}}\lim_{\varepsilon \to 0} \Phi_{1}\left(\lambda|F(z)-F^{(\varepsilon)}(z)|\right)dV_{\alpha}(z)=0. $$
We conclude that $F^{(\varepsilon)} \in A^{\Phi_{1}}_{\alpha}(\mathbb{C_{+}})\cap A^{\Phi_{2}}_{\beta}(\mathbb{C_{+}})$ and 
$$  \lim_{\varepsilon \to 0}\|F^{(\varepsilon)}-F\|_{A^{\Phi}_{\alpha}}^{lux} =0.  $$
The proof is complete.
\end{proof}

Let $\alpha>-1$. The Bergman kernel on $\mathbb{C}_{+}$ is defined by 
\begin{equation}\label{eq:fo4nAQ1l}
K_{\alpha}(z,\omega) =\left(\frac{z-\overline{\omega}}{i}\right)^{-\alpha-2}, ~~ \forall~ z, \omega \in \mathbb{C}_{+}.
\end{equation}
Let  $\alpha>-1$ and $F$  a measurable function on $\mathbb{C}_{+}$. The Bergman projection and positive Bergman operator of $F$ are respectively defined by
$$  \mathcal{P}_{\alpha}(F)(z):= \int_{\mathbb{C}_{+}}K_{\alpha}(z,\omega)F(\omega)dV_{\alpha}(\omega), ~~ \forall~z \in \mathbb{C}_{+}  $$
and
$$  \mathcal{P}_{\alpha}^{+}(F)(z):= \int_{\mathbb{C}_{+}}|K_{\alpha}(z,\omega)|F(\omega)dV_{\alpha}(\omega), ~~ \forall~z \in \mathbb{C}_{+}.  $$

\begin{proposition}\label{pro:mainplaq6}
Let   $\alpha>-1$ and $\Phi \in \mathscr{U} \cap \nabla_{2}$. For $z \in \mathbb{C}_{+}$, the  Bergman kernel $K_{\alpha}(z,.)\in A^{\Phi}_{\alpha}(\mathbb{C_{+}})$. 
\end{proposition} 

\begin{proof}
The proof follows from Lemma \ref{pro:main26} and Lemma \ref{pro:main27}.
\end{proof}

Let $\alpha>-1$, and let $f$ and $g$ be two measurable functions on  $\mathbb{C}_{+}$. We define the duality pairing associated with the weighted measure $dV_{\alpha}$ by
\begin{equation}\label{eq:fo4nl2}
\langle f,g\rangle_{\alpha}
:=
\int_{\mathbb{C}_{+}} f(z)\overline{g(z)}\,dV_{\alpha}(z),
\end{equation}
whenever this integral is well defined.

In the following result, we recall a pointwise estimate for the positive Bergman operator.

\begin{proposition}[Proposition 3.4, \cite{PottReguera}]\label{pro:mainplaqaaqaaqkq6}
There exists a constant \(C>0\) such that, for every $
F\in L^{1}_{\mathrm{loc}}(\mathbb{C}_{+},dV_{\alpha})$ and every \(z\in\mathbb{C}_{+}\),
\begin{equation}\label{eq:fo4nl}
\mathcal{P}_{\alpha}^{+}(|F|)(z)
\leq
C\sum_{\beta\in\{0,1/3\}} \sum_{I\in\mathcal{D}^{\beta}}
\left\langle |F|,\frac{\chi_{Q_{I}}}{|I|^{2+\alpha}}\right\rangle_{\alpha}
\chi_{Q_{I}}(z),
\end{equation}
where \(\langle\cdot,\cdot\rangle_{\alpha}\) denotes the duality pairing
associated with the weighted measure \(dV_{\alpha}\) defined in
\eqref{eq:fo4nl2}.
\end{proposition}

The preceding pointwise estimate provides the key estimate needed to transfer boundedness properties from the dyadic setting to continuous  setting. 

We can now state the main boundedness result.


\begin{theorem}\label{pro:mainplaqaaqaaqq6}
Let   $\alpha>-1$ and $\Phi \in \mathscr{U} \cap \nabla_{2}$. The following assertions are satisfied:
\begin{itemize}
\item[(i)]  $\mathcal{P}_{\alpha}: L^{\Phi}(\mathbb{C}_{+}, dV_{\alpha}) \longrightarrow A_{\alpha}^{\Phi}(\mathbb{C}_{+})$ is bounded.
\item[(ii)]  $\mathcal{P}_{\alpha}^{+}: L^{\Phi}(\mathbb{C}_{+}, dV_{\alpha})\longrightarrow L^{\Phi}(\mathbb{C}_{+}, dV_{\alpha})$ is bounded.
\end{itemize}
\end{theorem} 

\begin{proof}
Since 
$$ \vert \mathcal{P}_{\alpha}(F)(z)\vert \leq \mathcal{P}_{\alpha}^{+}(\vert F\vert)(z), ~~ \forall~ z \in \mathbb{C}_{+}, ~~ \forall~F \in L^{\Phi}(\mathbb{C}_{+}, dV_{\alpha}),   $$
it will suffice to prove that $\mathcal{P}_{\alpha}^{+}$ is bounded from $L^{\Phi}(\mathbb{C}_{+}, dV_{\alpha})$ to $L^{\Phi}(\mathbb{C}_{+}, dV_{\alpha})$ 
\medskip

Let $F \in L^{\Phi}(\mathbb{C}_{+}, dV_{\alpha})$. Since $L^{\Phi}(\mathbb{C}_{+}, dV_{\alpha}) \subset L^1_{\mathrm{loc}}(\mathbb{C}_{+}, dV_{\alpha})$,  Proposition \ref{pro:mainplaqaaqaaqkq6} implies that there exists a constant $C>0$,
independent of $F$, such that, for every  $z\in \mathbb{C}_{+}$,
$$ \mathcal{P}_{\alpha}^{+}(\vert F\vert)(z)
\leq
C\sum_{\beta\in\{0,1/3\}} \sum_{I\in\mathcal{D}^{\beta}}
\left\langle \vert F\vert,\frac{\chi_{Q_{I}}}{\vert I\vert^{2+\alpha}}\right\rangle_{\alpha}
\chi_{Q_{I}}(z).
   $$
Let $\Psi$ denote the complementary function of $\Phi$, and let
$G\in L^{\Psi}(\mathbb{C}_{+},dV_{\alpha})$ satisfy $\|G\|_{L_{\alpha}^{\Psi}}^{\mathrm{lux}}\leq 1$. Using the preceding pointwise estimate, we obtain
\begin{align*}
\left\vert\int_{\mathbb{C}_{+}}\mathcal{P}_{\alpha}^{+}(\vert F\vert)(z)\overline{G(z)}\, dV_{\alpha}(z)\right\vert
&\lesssim \int_{\mathbb{C}_{+}} \sum_{\beta\in\{0,1/3\}} \sum_{I\in\mathcal{D}^{\beta}}
\left\langle \vert F\vert,\frac{\chi_{Q_{I}}}{\vert I\vert^{2+\alpha}}\right\rangle_{\alpha}
\chi_{Q_{I}}(z)\vert G(z)\vert dV_{\alpha}(z) \\
&\lesssim \sum_{\beta\in\{0,1/3\}} \sum_{I\in\mathcal{D}^{\beta}}\int_{Q_{I}} 
\left\langle \vert F\vert,\frac{\chi_{Q_{I}}}{\vert I\vert^{2+\alpha}}\right\rangle_{\alpha}
\vert G(z)\vert dV_{\alpha}(z) \\
&\lesssim \sum_{\beta\in\{0,1/3\}} \sum_{I\in\mathcal{D}^{\beta}}\vert Q_{I}\vert_{\alpha}\left(\frac{1}{\vert Q_{I}\vert_{\alpha}}\int_{Q_{I}} \vert F(\omega)\vert\, dV_{\alpha}(\omega) \right)\left(\frac{1}{\vert Q_{I}\vert_{\alpha}} \int_{Q_{I}}\vert G(z)\vert\,  dV_{\alpha}(z) \right). 
\end{align*}
For each $\beta\in\left\{0,\frac{1}{3}\right\}$, since the family
  $\{T_{Q_{I}}\}_{I\in\mathcal{D}^{\beta}}$ consists of pairwise disjoint sets and  
  $$  \vert Q_I\vert_{\alpha}\sim \vert T_I\vert_{\alpha}, \forall I \in \mathcal{D}^{\beta},  $$
it follows that
\begin{align*}
\int_{\mathbb{C}_{+}}\left\vert \mathcal{P}_{\alpha}^{+}(F)(z)\overline{G(z)}\right\vert \, dV_{\alpha}(z)
&\lesssim \sum_{\beta\in\{0,1/3\}} \sum_{I\in\mathcal{D}^{\beta}}\int_{T_{I}}\left(\frac{1}{\vert Q_{I}\vert_{\alpha}}\int_{Q_{I}} \vert F(\omega)\vert dV_{\alpha}(\omega) \right)\left(\frac{1}{\vert Q_{I}\vert_{\alpha}} \int_{Q_{I}}\vert G(z)\vert\, dV_{\alpha}(z) \right)dV_{\alpha}(\zeta) \\
&\lesssim \sum_{\beta\in\{0,1/3\}} \int_{\bigcup_{I\in\mathcal{D}^{\beta}}T_{I}}\mathcal{M}_{\alpha}^{\mathcal{D}^{\beta}}(F)(\zeta)\mathcal{M}_{\alpha}^{\mathcal{D}^{\beta}}(G)(\zeta)dV_{\alpha}(\zeta) \\
&\lesssim \sum_{\beta\in\{0,1/3\}}\int_{\mathbb{C}_{+}}\mathcal{M}_{\alpha}^{\mathcal{D}^{\beta}}(F)(\zeta)\mathcal{M}_{\alpha}^{\mathcal{D}^{\beta}}(G)(\zeta)dV_{\alpha}(\zeta).
\end{align*}
Since  $\Phi \in \mathscr{U} \cap \nabla_{2}$, its complementary function also satisfies   $\Psi \in \mathscr{U} \cap \nabla_{2}$. 
Hence, by Proposition~\ref{pro:mainplaqaqq6}, for $\beta\in\{0,1/3\}$, the maximal operator
$\mathcal{M}_{\alpha}^{\mathcal{D}^{\beta}}$ is bounded on both
$L^{\Phi}(\mathbb{C}_{+}, dV_{\alpha})$ and
$L^{\Psi}(\mathbb{C}_{+}, dV_{\alpha})$. Therefore, by the
Orlicz--Hölder inequality,
\begin{align*}
\int_{\mathbb{C}_{+}}\mathcal{M}_{\alpha}^{\mathcal{D}^{\beta}}(F)(\zeta)\mathcal{M}_{\alpha}^{\mathcal{D}^{\beta}}(G)(\zeta)dV_{\alpha}(\zeta)
&\lesssim \|\mathcal{M}_{\alpha}^{\mathcal{D}^{\beta}}(F)\|_{L_{\alpha}^{\Phi}}^{lux} \|\mathcal{M}_{\alpha}^{\mathcal{D}^{\beta}}(G)\|_{L_{\alpha}^{\Psi}}^{lux} \\
&\lesssim \|F\|_{L_{\alpha}^{\Phi}}^{lux} \|G\|_{L_{\alpha}^{\Psi}}^{lux} 
\lesssim \|F\|_{L_{\alpha}^{\Phi}}^{lux}.
\end{align*}
Combining the preceding estimates, we conclude that
$$  \left\vert\int_{\mathbb{C}_{+}}\mathcal{P}_{\alpha}^{+}(F)(z)\overline{G(z)}\,dV_{\alpha}(z)\right\vert \lesssim \|F\|_{L_{\alpha}^{\Phi}}^{lux},
   $$
which completes the proof.
\end{proof}

\begin{lemma}\label{pro:mainplaqzepama6}
Let   $\alpha>-1$ and   $\Phi \in  \mathscr{U}\cap \nabla_{2}$. For $F\in A^{\Phi}_{\alpha}(\mathbb{C_{+}})$, we have
$$ F(z)= \mathcal{P}_{\alpha}(F)(z), ~~ \forall~z \in \mathbb{C}_{+}.  $$
\end{lemma}

\begin{proof}
For $F\in A^{\Phi}_{\alpha}(\mathbb{C_{+}})$, there exists a sequence $\{F_{n}\}_{n}$ in $A^{2}_{\alpha}(\mathbb{C_{+}}) \cap A^{\Phi}_{\alpha}(\mathbb{C_{+}})$ such that $$ \lim_{n \to \infty}\|F_{n}-F\|_{A_{\alpha}^{\Phi}}^{lux}=0,    $$
according to Lemma \ref{pro:main11aqaaqm8}. For $n \geq 1$, since $F_{n} \in A^{2}_{\alpha}(\mathbb{C_{+}})$,  we have 
$$  F_{n}(z)=\mathcal{P}_{\alpha}(F_{n})(z) , ~~ \forall~z \in \mathbb{C}_{+}.    $$
For $z \in \mathbb{C}_{+}$, we have 
\begin{align*}
|F(z)-\mathcal{P}_{\alpha}(F)(z)|&= |F(z)-F_{n}(z)+\mathcal{P}_{\alpha}(F_{n})(z)-\mathcal{P}_{\alpha}(F)(z)| \\
&\leq |F(z)-F_{n}(z)| + |\mathcal{P}_{\alpha}(F-F_{n})(z)| \\
&\leq |F(z)-F_{n}(z)| + 2\|K_{\alpha}(z,.)\|_{A_{\alpha}^{\Psi}}^{lux} \|F-F_{n}\|_{A_{\alpha}^{\Phi}}^{lux}, 
\end{align*}
where  $\Psi$ is the complementary function of $\Phi$. Taking the limit when $n\to\infty$, we deduce that 
$ F(z)=\mathcal{P}_{\alpha}(F)(z)$.
\end{proof}

\begin{theorem}\label{pro:main 5aqk3pl}
Let   $\alpha>-1$ and $\Phi \in \mathscr{U} \cap \nabla_{2}$.  The topological dual   $\left(A_{\alpha}^{\Phi}(\mathbb{C}_{+})\right)^{*}$ of $A_{\alpha}^{\Phi}(\mathbb{C}_{+})$     is isomorphic to $A_{\alpha}^{\Psi}(\mathbb{C}_{+})$, in the sense that, for all  $T\in \left(A_{\alpha}^{\Phi}(\mathbb{C}_{+})\right)^{*}$, there is a unique $G\in A_{\alpha}^{\Psi}(\mathbb{C}_{+})$ such that
 $$ T(F) = T_{G}(F):= \int_{0}^{\infty}\left(\int_{\mathbb{R}}F(x+iy)\overline{G(x+iy)}dx\right)y^{\alpha}dy, ~~ \forall~F\in A_{\alpha}^{\Phi}(\mathbb{C}_{+}), 
   $$
where $\Psi$ is the complementary function of $\Phi$.
\end{theorem}

\begin{proof}
Let  $T\in \left(A_{\alpha}^{\Phi}(\mathbb{C}_{+})\right)^{*}$. By Hahn Banach's theorem, $T$ extends to a bounded linear functional on $L^{\Phi}(\mathbb{C}_{+},dV_{\alpha})$
of the same norm. Since $\left(L^{\Phi}(\mathbb{C}_{+},dV_{\alpha})\right)^{*}=L^{\Psi}(\mathbb{C}_{+},dV_{\alpha})$, there exists a unique function $\widetilde{G}\in L^{\Psi}(\mathbb{C}_{+},dV_{\alpha})$ such that for $F\in A_{\alpha}^{\Phi}(\mathbb{C}_{+})$, 
 $$  T(F)= \int_{0}^{\infty}\left(\int_{\mathbb{R}}F(x+iy)\overline{\widetilde{G}(x+iy)}dx\right)y^{\alpha}dy     $$
and  $\|T\|_{(L_{\alpha}^{\Phi})^{*}}= \|\widetilde{G}\|_{L_{\alpha}^{\Psi}}$ (see \cite{raoren}).
Since $\mathcal{P}_{\alpha}$  is self-adjoint and $F=\mathcal{P}_{\alpha}(F)$, it follows that 
 $$  T(F)= \int_{0}^{\infty}\left(\int_{\mathbb{R}}\mathcal{P}_{\alpha}(F)(x+iy)\overline{\widetilde{G}(x+iy)}dx\right)y^{\alpha}dy = \int_{0}^{\infty}\left(\int_{\mathbb{R}}F(x+iy)\overline{\mathcal{P}_{\alpha}(\widetilde{G})(x+iy)}dx\right)y^{\alpha}dy,   $$
with $\mathcal{P}_{\alpha}(\widetilde{G}) \in A_{\alpha}^{\Psi}(\mathbb{C}_{+})$. Therefore, 
$   T=T_{G}$, with $G=\mathcal{P}_{\alpha}(\widetilde{G})$.
\end{proof}

\subsection{Orlicz space of sequences on  $\mathbb{C}_{+}$.}

Let $0<\delta< 1$ and   $\gamma$ the real number defined in (\ref{eq:uaa5aqn}). For $l,j\in \mathbb{Z}$, put 
\begin{align*}
I_{l,j}&:= \left] \frac{\delta^{2}}{4}\left(-1+\frac{l}{2} \right)2^{\gamma j};  \frac{\delta^{2}}{4}\left(1+\frac{l}{2} \right)2^{\gamma j}   \right[ \\
I_{l,j}'&:= \left] \frac{\delta^{2}}{4}\left(-\frac{1}{5}+\frac{l}{2} \right)2^{\gamma j};  \frac{\delta^{2}}{4}\left(\frac{1}{5}+\frac{l}{2} \right)2^{\gamma j}   \right[ \\
J_{j}&:= \left] \left(1-\frac{\delta^{2}}{4} \right)2^{\gamma j};  \left(1+\frac{\delta^{2}}{4} \right)2^{\gamma j}   \right[ \\
J_{j}'&:= \left] \left(1-\frac{\delta^{2}}{20} \right)^{1/4}2^{\gamma j};  \left(1+\frac{\delta^{2}}{20} \right)^{1/4}2^{\gamma j}   \right[. 
\end{align*}
We have
\begin{equation}\label{eq:5aaqn}
 |J_{j}|=|I_{l,j}|=\frac{\delta^{2}}{2}y_{l,j}, \hspace*{0.5cm} |I_{l,j}'|=\frac{\delta^{2}}{10}y_{l,j} \hspace*{0.5cm}\textrm{and} \hspace*{0.5cm} |J_{j}'|= C_{\delta}y_{l,j},
\end{equation}
where $C_{\delta}$ is a constant which depends only on $\delta$.

\medskip

For $y>0$ and $x \in \mathbb{R}$,  and  $j\in \mathbb{Z}$, put 
\begin{equation}\label{eq:5an}
L=L(y):=\{ j\in \mathbb{Z}: y \in J_{j}   \} \hspace*{0.5cm}\textrm{and} \hspace*{0.5cm} L_{j}=L_{j}(x):=\{ l\in \mathbb{Z}: x \in I_{l,j}   \}.
\end{equation}
We have 
\begin{equation}\label{eq:inpmpy}
\text{Card} (L_{j}(x)) \leq 4 \hspace*{1cm}\textrm{and} \hspace*{1cm} \text{Card} (L(y)) \leq N_{\delta},
\end{equation}
where  $\text{Card}(A)$ is the number of elements contained in the set $A$ and 
$N_{\delta} \geq 1$ is an integer  which depends only on $\delta$. The following properties hold.
 \begin{itemize}
 \item[(i)] $( 0, \infty) =\displaystyle\bigcup_{j\in \mathbb{Z}}J_{j}$ and   $\mathbb{R}=\cup_{l\in \mathbb{Z}}I_{l,j}$, for $j\in \mathbb{Z}$ fixed;
 \item[(ii)] For  $j\in \mathbb{Z}$, $J_{j}' \subset J_{j}$ and    $J_{j+1}'\cap J_{j}'=\emptyset$; 
 \item[(iii)] For fixed  $j\in \mathbb{Z}$,  $I_{l,j}' \subset I_{l,j}$ and  $I_{l,j}'\cap I_{l+1,j}'=\emptyset$, for all  $l\in \mathbb{Z}$;
 \item[(iv)] For fixed  $j\in \mathbb{Z}$, each element  $x\in \mathbb{R}$ is in at most four intervals $I_{l,j}$;
 \item[(v)] There exists an integer $N$ such that any $y \in ] 0; \infty [$ is in at most $N$ intervals $J_{j}$.
 \end{itemize}

For $0<\delta< 1$, put 
 \begin{equation}\label{eq:uaa5n}
  s_{\delta}:=-1+ \left( 1+\delta^{2}/20\right)^{1/4}.
 \end{equation}
 We also have the following properties:
 \begin{itemize}
 \item[(a)] $\mathbb{C}_{+}=\displaystyle\bigcup_{l,j\in \mathbb{Z}} \mathcal{D}_{\delta}(z_{l,j})$;
\item[(b)] For  $l,j\in \mathbb{Z}$, $\mathcal{D}_{s_{\delta}}(z_{l,j}) \subset \mathcal{D}_{\delta}(z_{l,j})$ and $\mathcal{D}_{s_{\delta}}(z_{l,j}) \cap \mathcal{D}_{s_{\delta}}(z_{k,p}) =\emptyset$, for all $(l,j)\not=(k,p)$;
 \item[(c)] There is an integer  $N$ (depending only on $\delta$) such that each point of
 $\mathbb{C}_{+}$ belongs to at most $N$ of the disks $\mathcal{D}_{\delta}(z_{l,j})$.
\end{itemize}

The preceding geometric properties allow us to establish a connection between the Bergman--Orlicz norm of an analytic function and the discrete sequence of its values at the points \(z_{l,j}\). We can therefore state the following sampling result.

\begin{lemma}\label{pro:mainqaqa6}
Let $\alpha>-1$,  $0<\delta<1$,    $\Phi \in \mathscr{L}\cup\mathscr{U}$ and  $\{z_{l,j}\}_{l,j\in \mathbb{Z}}$, the sequence defined in (\ref{eq:ua5n}).  For  $F\in A_{\alpha}^{\Phi}(\mathbb{C}_{+})$,  the sequence $\{F(z_{l,j})\}_{l,j\in \mathbb{Z}}$ belongs to $\ell^{\Phi}_{\alpha}(\mathbb{C_{+}})$.  Moreover,
\begin{equation}\label{eq:inaqqay}
\left\|\{F(z_{l,j})\}_{l,j\in \mathbb{Z}}\right\|_{\ell^{\Phi}_{\alpha}}^{lux}\leq C_{1} \|F\|_{A^{\Phi}_{\alpha}}^{lux}, 
 \end{equation}
where $C_{1}$ is a positive constant depending only on  $\delta$ and $\alpha$.

Conversely, for $\delta$  small enough, we have 
\begin{equation}\label{eq:inay}
 \|F\|_{A^{\Phi}_{\alpha}}^{lux}\leq C_{2}\left\|\{F(z_{l,j})\}_{l,j\in \mathbb{Z}}\right\|_{\ell^{\Phi}_{\alpha}}^{lux},
 \end{equation}
where  $C_{2}$ is a positive constant depending only on $\Phi$, $\delta$ and $\alpha$.
\end{lemma}

\begin{proof}
We may assume that $F\not\equiv 0$, since there is nothing to prove when
$F\equiv 0$. We first prove \eqref{eq:inaqqay}.

Let $l,j\in\mathbb{Z}$. By Proposition~\ref{pro:main1q8}, we have
\begin{equation}\label{eq:proof1}
\Phi\left(\frac{|F(z_{l,j})|}{\|F\|_{A^{\Phi}_{\alpha}}^{lux}}\right)\lesssim \frac{1}{\left( \mathrm{Im}( z_{l,j})\right)^{2+\alpha}} \int_{\mathcal{D}_{s_{\delta}}\left(z_{l,j}\right)}\Phi\left(\frac{|F(u+iv)|}{\|F\|_{A^{\Phi}_{\alpha}}^{lux}}\right)dV_{\alpha}(u+iv)
\end{equation}
where $s_{\delta}$ is the constant defined in \eqref{eq:uaa5n}.

Since, by the definition of the sequence $\{z_{l,j}\}_{l,j\in\mathbb{Z}}$, $\operatorname{Im}(z_{l,j})=2^{j\gamma},
$ we obtain
$$
\big(\operatorname{Im}(z_{l,j})\big)^{\alpha+2}
=2^{j\gamma(\alpha+2)}.
$$
Multiplying \eqref{eq:proof1} by
$2^{j\gamma(\alpha+2)}$, it follows that
\begin{equation}\label{eq:proof2}
\Phi\left(\frac{|F(z_{l,j})|}{\|F\|_{A^{\Phi}_{\alpha}}^{lux}}\right)2^{j\gamma (\alpha +2)} \lesssim  \int_{\mathcal{D}_{s_{\delta}}\left(z_{l,j}\right)}\Phi\left(\frac{|F(u+iv)|}{\|F\|_{A^{\Phi}_{\alpha}}^{lux}}\right)dV_{\alpha}(u+iv).
\end{equation}
By construction, the family $\{\mathcal{D}_{s_{\delta}}\left(z_{l,j}\right)\}_{l,j\in \mathbb{Z}}$ consists of pairwise disjoint sets. Hence, summing \eqref{eq:proof2}
over $l,j\in\mathbb{Z}$ yields
\begin{align*}
\sum_{j\in \mathbb{Z}}\sum_{l\in \mathbb{Z}}\Phi\left(\frac{|F(z_{l,j})|}{\|F\|_{A^{\Phi}_{\alpha}}^{lux}}\right)2^{j\gamma (\alpha +2)}
&\lesssim  \sum_{j\in \mathbb{Z}}\sum_{l\in \mathbb{Z}} \int_{\mathcal{D}_{s_{\delta}}\left(z_{l,j}\right)}\Phi\left(\frac{|F(u+iv)|}{\|F\|_{A^{\Phi}_{\alpha}}^{lux}}\right)dV_{\alpha}(u+iv) \\
&\lesssim \int_{\mathbb{C_{+}}}\Phi\left(\frac{|F(u+iv)|}{\|F\|_{A^{\Phi}_{\alpha}}^{lux}}\right)dV_{\alpha}(u+iv) 
\lesssim  1.
\end{align*}
Therefore,  $\{F(z_{l,j})\}_{l,j\in \mathbb{Z}}$ belongs to $\ell^{\Phi}_{\alpha}(\mathbb{C_{+}})$ and 
$$ \left\|\{F(z_{l,j})\}_{l,j\in \mathbb{Z}}\right\|_{\ell^{\Phi}_{\alpha}}^{lux}\lesssim \|F\|_{A^{\Phi}_{\alpha}}^{lux}. 
  $$

\medskip

We now turn to the proof of \eqref{eq:inay}.

\medskip

We may assume, without loss of generality, that $\left\|\{F(z_{l,j})\}_{l,j\in \mathbb{Z}}\right\|_{\ell^{\Phi}_{\alpha}}^{lux}=1$. 

By the quasi-subadditivity property of $\Phi$, there exists a constant
$K_{\Phi}>0$, depending only on $\Phi$, such that
\begin{align*}
\int_{0}^{+\infty}\int_{-\infty}^{+\infty}\Phi\left( |F(x+iy)|\right)y^{\alpha}dxdy 
&\leq K_{\Phi}\sum_{j\in \mathbb{Z}}\sum_{l\in \mathbb{Z}}\int_{J_{j}}\int_{I_{l,j}}\Phi\left( |F(z_{l,j})|  \right)y^{\alpha}dxdy \\
&+ K_{\Phi}\sum_{j\in \mathbb{Z}}\sum_{l\in \mathbb{Z}}\int_{J_{j}}\int_{I_{l,j}}\Phi\left( |F(x+iy)-F(z_{l,j})|  \right)y^{\alpha}dxdy.
\end{align*}
We denote the two terms on the right-hand side by
 $$  I:=\sum_{j\in \mathbb{Z}}\sum_{l\in \mathbb{Z}}\int_{J_{j}}\int_{I_{l,j}}\Phi\left( |F(z_{l,j})|  \right)y^{\alpha}dxdy  $$
and 
$$  II:=\sum_{j\in \mathbb{Z}}\sum_{l\in \mathbb{Z}}\int_{J_{j}}\int_{I_{l,j}}\Phi\left( |F(x+iy)-F(z_{l,j})|  \right)y^{\alpha}dxdy.   $$

\medskip

Let us give an estimate of $I$. 

\medskip

We have
\begin{align*}
\sum_{j\in \mathbb{Z}}\sum_{l\in \mathbb{Z}}\int_{J_{j}}\int_{I_{l,j}}\Phi\left( |F(z_{l,j})|  \right)y^{\alpha}dxdy  
&=\sum_{j\in \mathbb{Z}}\sum_{l\in \mathbb{Z}}\Phi\left( |F(z_{l,j})|  \right)\int_{J_{j}}\int_{I_{l,j}}y^{\alpha}dxdy \\
&\lesssim \sum_{j\in \mathbb{Z}}\sum_{l\in \mathbb{Z}}\Phi\left( |F(z_{l,j})| \right)2^{\gamma j(\alpha+2)} \lesssim  1.
\end{align*}
We deduce that
\begin{equation}\label{eq:inqay}
I:=\sum_{j\in \mathbb{Z}}\sum_{l\in \mathbb{Z}}\int_{J_{j}}\int_{I_{l,j}}\Phi\left( |F(z_{l,j})|  \right)y^{\alpha}dxdy\lesssim 1.
\end{equation}

\medskip

Let us also give an estimate of $II$. The objective is to prove that 
\begin{equation}\label{eq:inqaaqy}
\sum_{j\in \mathbb{Z}}\sum_{l\in \mathbb{Z}}\int_{J_{j}}\int_{I_{l,j}}\Phi\left( |F(x+iy)-F(z_{l,j})|  \right)y^{\alpha}dxdy \leq C \int_{0}^{+\infty}\int_{-\infty}^{+\infty}\Phi\left( |F(x+iy)|\right)y^{\alpha}dxdy,  
\end{equation}
where $C>0$ is a constant, which will be chosen sufficiently small so that  $0<C <1$.

\medskip

Fix $l,j\in\mathbb{Z}$. For $0<\delta <1$, we have
\begin{align*}
\int_{J_{j}'}y^{\alpha}dy&=\frac{2^{\gamma j(\alpha+1)}}{\alpha+1}\left(1+\delta^{2}/20 \right)^{(\alpha+1)/4}\left( 1- \left( \frac{1-\delta^{2}/20}{1+\delta^{2}/20}  \right)^{(\alpha+1)/4}       \right) \\
&\geq \frac{2^{\gamma j(\alpha+1)}}{\alpha+1}\left( 1- \left( \frac{1-\delta^{2}/20}{1+\delta^{2}/20}  \right)^{(\alpha+1)/4}       \right)
\end{align*}
and 
\begin{align*}
\int_{J_{j}}y^{\alpha}dy&=\frac{2^{\gamma j(\alpha+1)}}{\alpha+1}\left(1+\delta^{2}/4 \right)^{\alpha+1}\left( 1- \left( \frac{1-\delta^{2}/4}{1+\delta^{2}/4}  \right)^{\alpha+1}       \right) \\
&\leq\frac{2^{\gamma j(\alpha+1)}}{\alpha+1}\times 2^{\alpha+1}\left( 1- \left( \frac{1-\delta^{2}/4}{1+\delta^{2}/4}  \right)^{\alpha+1}       \right). 
\end{align*}
We deduce that
\begin{equation}\label{eq:inqaaaqy}
\int_{J_{j}}y^{\alpha}dy \leq  C_{1} \int_{J_{j}'}y^{\alpha}dy,
\end{equation}
where 
\begin{equation}\label{eq:inqaaaqy7}
C_{1}:= 2^{\alpha+1}\left( 1- \left( \frac{1-\delta^{2}/4}{1+\delta^{2}/4}  \right)^{\alpha+1}       \right)\left( 1- \left( \frac{1-\delta^{2}/20}{1+\delta^{2}/20}  \right)^{(\alpha+1)/4}       \right)^{-1}.
\end{equation}
We also have
\begin{equation}\label{eq:inqaaqaqy}
|I_{l,j}|=\frac{\delta^{2}}{2}y_{l,j}=5|I_{l,j}'|,
\end{equation}
thanks to relation (\ref{eq:5aaqn}).  

\medskip

Assume that $\Phi$ is of lower type $p\in(0,\infty)$.  Let $0<\delta<-1+\frac{\sqrt{5}}{2}$, and  define
\[
\delta_{1}:=\delta(2+\delta)
\qquad\text{and}\qquad
\delta_{2}:=\frac{3\delta_{1}}{1-\delta_{1}}.
\]
Since  $0<\delta<-1+\frac{\sqrt{5}}{2}$, we deduce that $0<\delta_{1}<1/4$ and $0<\delta_{2}<1$. Moreover, $$  0<\delta<\delta_{1}<\delta_{2}<1. $$

\medskip

We next estimate the oscillation of $F$ on
$I_{l,j}\times J_j$. 

\medskip

For every $x\in I_{l,j}$ and $y\in J_j$, the definitions of $I_{l,j}$ and $J_j$ give
\begin{align*}
|x+iy- z_{l,j}| &\leq |x- x_{l,j}|+|y- y_{l,j}| \\
 &< \frac{\delta^{2}}{4}y_{l,j}+ \frac{\delta^{2}}{4}y_{l,j}= \frac{\delta^{2}}{2}y_{l,j}.\end{align*}
Therefore, applying Proposition~\ref{pro:main1q8} with
$s=\delta$, $s'=\delta_1$, and $n=2$, we obtain
\begin{equation}\label{eq:inqaaqy1}
\Phi\left( |F(x+iy)-F(z_{l,j})|  \right) \leq C_{p}\delta^{p-2}\left( 1+ \delta   \right)^{2} \int_{\mathcal{D}_{\delta_{1}}(z_{l,j})}\Phi\left( |F(u+iv)|  \right)\frac{dudv}{v^{2}},
\end{equation}
where $C_p$ is a constant depending only on $p$. Moreover, since  
$$  |x+iy- z_{l,j}|< \frac{\delta^{2}}{2}y_{l,j} < \delta_1 y_{l,j},$$
we have $x+iy \in \mathcal{D}_{\delta_{1}}(z_{l,j})$. Consequently, by (\ref{eq:espeqs1})
$$ \left(1- \delta_1 \right) \mathrm{Im}( z_{l,j}) < y < \left(1+ \delta_1 \right) \mathrm{Im}( z_{l,j}).
   $$
We claim that the disk $\mathcal{D}_{\delta_{1}}(z_{l,j})$ is contained in disk $\mathcal{D}_{\delta_{2}}(x+iy)$. Indeed, for $u+iv \in \mathcal{D}_{\delta_{1}}(z_{l,j})$, we have
\begin{align*}
|x+iy- (u+iv)| &\leq |x- u|+|y- v| \\
&\leq |x- x_{l,j}|+|u- x_{l,j}|+|y- y_{l,j}|+|y_{l,j}- v| \\
&< \frac{\delta^{2}}{4}y_{l,j}+\delta_{1}y_{l,j}+\frac{\delta^{2}}{4}y_{l,j}+\delta_{1}y_{l,j} \\
&< \delta y_{l,j} + 2\delta_{1}y_{l,j} \\
&< 3\delta_{1}y_{l,j} \\
&< \frac{3\delta_{1}}{\left(1- \delta_1 \right)}y =\delta_{2}y.
\end{align*}
Combining (\ref{eq:inqaaaqy}), (\ref{eq:inqaaqaqy}), and (\ref{eq:inqaaqy1}), we obtain, for  $0<\delta<-1+\frac{\sqrt5}{2}$,
\begin{align*}
II(l,j)&:=\int_{J_{j}}\int_{I_{l,j}}\Phi\left( |F(x+iy)-F(z_{l,j})|  \right)y^{\alpha}dxdy \\
&\leq  \int_{J_{j}}\int_{I_{l,j}}\left(C_{p}\delta^{p-2}\left( 1+ \delta   \right)^{2} \int_{\mathcal{D}_{\delta_{1}}(z_{l,j})}\Phi\left( |F(u+iv)|  \right)\frac{dudv}{v^{2}}\right)y^{\alpha}dxdy\\
&\leq 5C_{1}C_{p}\delta^{p-2}\left( 1+ \delta   \right)^{2} \int_{J_{j}'}\int_{I_{l,j}'}\left(\int \int_{\mathcal{D}_{\delta_{1}}(z_{l,j})}\Phi\left( |F(u+iv)|  \right)\frac{dudv}{v^{2}}\right)y^{\alpha}dxdy\\
&\leq 5C_{1}C_{p}\delta^{p-2}\left( 1+ \delta   \right)^{2} \int_{J_{j}'}\int_{I_{l,j}'}\left(\int \int_{\mathcal{D}_{\delta_{2}}(x+iy)}\Phi\left( |F(u+iv)|  \right)\frac{dudv}{v^{2}}\right)y^{\alpha}dxdy\\
&\leq 5C_{1}C_{p}\delta^{p-2}\left( 1+ \delta   \right)^{2} \int_{J_{j}'}\int_{I_{l,j}'}\left(\int \int_{|x- u| \leq \delta_{2}y}\int_{|y- v| \leq \delta_{2}y}\Phi\left( |F(u+iv)|  \right)\frac{dudv}{v^{2}}\right)y^{\alpha}dxdy.
\end{align*}
It follows that, 
\begin{align*}
II&:= \sum_{j\in \mathbb{Z}}\sum_{l\in \mathbb{Z}}\int_{J_{j}}\int_{I_{l,j}}\Phi\left( |F(x+iy)-F(z_{l,j})|  \right)y^{\alpha}dxdy \\
&\leq 5C_{1}C_{p}\delta^{p-2}\left( 1+ \delta   \right)^{2}\sum_{j\in \mathbb{Z}}\sum_{l\in \mathbb{Z}}\int_{J_{j}'}\int_{I_{l,j}'}\left(  \int_{|x- u| \leq \delta_{2}y}\int_{|y- v| \leq \delta_{2}y}\Phi\left( |F(u+iv)|  \right)\frac{dudv}{v^{2}} \right)y^{\alpha}dxdy \\
&=5C_{1}C_{p}\delta^{p-2}\left( 1+ \delta   \right)^{2} \int_{\bigcup_{j\in \mathbb{Z}}J_{j}'}\int_{\bigcup_{l\in \mathbb{Z}}I_{l,j}'}\int_{(1-\delta_{2})y}^{(1+\delta_{2})y}\Phi\left( |F(u+iv)| \right) \left(\int_{u-\delta_{2}y}^{u+\delta_{2}y} dx \right) \frac{dudv}{v^{2}}y^{\alpha}dy \\
&\leq 10\delta_{2}C_{1}C_{p}\delta^{p-2}\left( 1+ \delta   \right)^{2}\int_{0}^{\infty}\int_{(1-\delta_{2})y}^{(1+\delta_{2})y}\left(\int_{-\infty}^{\infty}\Phi\left( |F(u+iv)| \right) du \right)\frac{dv}{v^{2}}y^{\alpha+1}dy \\
&= 10\delta_{2}C_{1}C_{p}\delta^{p-2}\left( 1+ \delta   \right)^{2}\int_{0}^{\infty}\int_{(1-\delta_{2})y}^{(1+\delta_{2})y}\|F(.+iv)\|_{L^{\Phi}}v^{-2}y^{\alpha+1}dvdy.
\end{align*}
Since the function $v\mapsto \|F(.+iv)\|_{L^{\Phi}}$ is non-increasing on $(0, \infty)$, we have
\begin{align*}
\int_{0}^{\infty}\int_{(1-\delta_{2})y}^{(1+\delta_{2})y}\|F(.+iv)\|_{L^{\Phi}}v^{-2}y^{\alpha+1}dvdy &\leq \int_{0}^{\infty}\|F(.+i(1-\delta_{2})y)\|_{L^{\Phi}}\left(\int_{(1-\delta_{2})y}^{(1+\delta_{2})y}v^{-2}dv\right)y^{\alpha+1}dy \\
&\leq \frac{2\delta_{2}}{(1-\delta_{2})^{2}} \int_{0}^{\infty}\|F(.+i(1-\delta_{2})y)\|_{L^{\Phi}}y^{\alpha}dy \\
&= \frac{2\delta_{2}}{(1-\delta_{2})^{3+\alpha}} \int_{0}^{\infty}\|F(.+iy)\|_{L^{\Phi}}y^{\alpha}dy.
\end{align*}
Combining the preceding estimates, we deduce that
\begin{equation}\label{eq:inqaaqym1}
II:\leq \frac{20\delta_{2}^{2}C_{1}C_{p}\delta^{p-2}\left( 1+ \delta   \right)^{2}}{(1-\delta_{2})^{3+\alpha}} \int_{0}^{\infty}\|F(.+iy)\|_{L^{\Phi}}y^{\alpha}dy.
\end{equation}
Therefore, after multiplying by the constant $K{\Phi}$ arising from the quasi-subadditivity of $\Phi$, we set
 $$  C_{\alpha,\delta,\Phi}:= \frac{20K_{\Phi}C_{1}C_{p}\delta^{p-2}\left( 1+ \delta   \right)^{2}\delta_{2}^{2}}{(1-\delta_{2})^{3+\alpha}}.  $$
For   $0<\delta<\frac{1}{2}\left(  -1+\frac{\sqrt{5}}{2} \right)$, using the explicit expressions for $C_{1}$ and $\delta_{2}$, we obtain
\begin{align*}
C_{\alpha,\delta,\Phi}&:= \frac{20K_{\Phi}C_{1}C_{p}\delta^{p-2}\left( 1+ \delta   \right)^{2}\delta_{2}^{2}}{(1-\delta_{2})^{3+\alpha}} \\
&= \frac{180\times 2^{\alpha+1} K_{\Phi}C_{p}\delta^{p}\left( 1+ \delta   \right)^{2}\left( 2+ \delta   \right)^{2}\left( 1-\delta(2+\delta) \right)^{1+\alpha}\left( 1- \left( \frac{1-\delta^{2}/4}{1+\delta^{2}/4}  \right)^{\alpha+1}       \right)}{(1-4\delta(2+\delta))^{3+\alpha}\left( 1- \left( \frac{1-\delta^{2}/20}{1+\delta^{2}/20}  \right)^{(\alpha+1)/4}       \right)} \\
&\leq K_{\Phi}C_{\alpha}C_{p}\delta^{p},
\end{align*}
where $C_{\alpha}$ is a positive constant depending only on $\alpha$. In particular, if  $$0<\delta<\min\left\{  \frac{1}{2}\left(  -1+\frac{\sqrt{5}}{2} \right); \left( \frac{1}{K_{\Phi}C_{\alpha}C_{p}} \right)^{1/p}   \right\}, $$
then $0<C_{\alpha,\delta,\Phi}<1$. 

We now combine the estimates obtained for $I$ and $II$. By (\ref{eq:inqay}) and (\ref{eq:inqaaqym1}), respectively, we have
\begin{align*}
\int_{0}^{+\infty}\int_{-\infty}^{+\infty}\Phi\left( |F(x+iy)|\right)y^{\alpha}dxdy 
&\leq K_{\Phi}( I +II   ) \\
&\lesssim 1+ C_{\alpha,\delta,\Phi}\int_{0}^{+\infty}\int_{-\infty}^{+\infty}\Phi\left( |F(x+iy)|\right)y^{\alpha}dxdy.
\end{align*}
Hence,
$$  (1-C_{\alpha,\delta,\Phi})\int_{0}^{+\infty}\int_{-\infty}^{+\infty}\Phi\left( |F(x+iy)|\right)y^{\alpha}dxdy \lesssim 1.   $$
This completes the proof.
\end{proof}

\begin{lemma}\label{pro:mainqaqaqq6}
Let $\alpha>-1$,  $0<\delta<1$,  $\Phi \in \mathscr{U} \cap \nabla_{2}$ and  $\mu:=\{\mu_{l,j}\}_{l,j\in \mathbb{Z}} \in \ell^{\Phi}_{\alpha}(\mathbb{C_{+}})$. For $l,j\in \mathbb{Z}$ and  $z \in \mathbb{C}_{+}$,  put 
\begin{equation}\label{eq:imay}
F_{l,j}(z) =2^{\alpha+2}\mu_{l,j}\left(  \frac{z-\overline{z_{l,j}}}{i} \right)^{-\alpha-2}2^{j\gamma (\alpha +2)}.
\end{equation}
Then the functions $F_{l,j}$ and $\sum_{j\in \mathbb{Z}}\sum_{l\in \mathbb{Z}}F_{l,j}$ belong to  $A_{\alpha}^{\Phi}(\mathbb{C}_{+})$. Moreover, 
\begin{equation}\label{eq:inaqqaqqay}
F_{l,j}(z_{l,j})=\mu_{l,j}
 \end{equation}
and
\begin{equation}\label{eq:ingaaaqqy}
 \left\|\sum_{j\in \mathbb{Z}}\sum_{l\in \mathbb{Z}}|F_{l,j}|\right\|_{L^{\Phi}_{\alpha}}^{lux} \sim \left\|\mu\right\|_{\ell^{\Phi}_{\alpha}}^{lux}.
 \end{equation}
\end{lemma}

\begin{proof}
Assume that $\mu\not\equiv 0$ because there is nothing to show when $\mu \equiv 0$. We can assume also without loss of generality that $\left\|\mu\right\|_{\ell^{\Phi}_{\alpha}}^{lux}=1$.

\medskip

Let  $ 0\not\equiv G \in L^{\Psi}(\mathbb{C}_{+}, dV_{\alpha})$ such that $\|G\|_{L^{\Psi}_{\alpha}}^{lux}= 1$, where  $\Psi$ is the complementary function of $\Phi$. For  $l,j\in \mathbb{Z}$, we have
\begin{align*}
\int_{\mathbb{C}_{+}}\sum_{j\in \mathbb{Z}}\sum_{l\in \mathbb{Z}}|F_{l,j}(z)G(z)|dV_{\alpha}(z) &\lesssim \sum_{j\in \mathbb{Z}}\sum_{l\in \mathbb{Z}} |\mu_{l,j}|2^{j\gamma (\alpha +2)}\int_{\mathbb{C}_{+}}|K_{\alpha}(z_{l,j},z)|G(z)|dV_{\alpha}(z) \\
&\lesssim \sum_{j\in \mathbb{Z}}\sum_{l\in \mathbb{Z}} |\mu_{l,j}|\mathcal{P}_{\alpha}^{+}(|G|)(z_{l,j})2^{j\gamma (\alpha +2)} \\
&\lesssim \sum_{j\in \mathbb{Z}}\sum_{l\in \mathbb{Z}}\Phi(|\mu_{l,j}|)2^{j\gamma (\alpha +2)} + \sum_{j\in \mathbb{Z}}\sum_{l\in \mathbb{Z}} \Psi\left(  |\mathcal{P}_{\alpha}^{+}(|G|)(z_{l,j}) \right)2^{j\gamma (\alpha +2)} \\
&\lesssim 1 + \sum_{j\in \mathbb{Z}}\sum_{l\in \mathbb{Z}} \Psi\left(  |\mathcal{P}_{\alpha}^{+}(|G|)(z_{l,j}) \right)2^{j\gamma (\alpha +2)},
\end{align*}
thanks to Young's inequality in Orlicz spaces. Let us now show that
$$  \sum_{j\in \mathbb{Z}}\sum_{l\in \mathbb{Z}} \Psi\left(  |\mathcal{P}_{\alpha}^{+}(|G|)(z_{l,j}) \right)2^{j\gamma (\alpha +2)}  < \infty.   $$
We have
\begin{align*}
\sum_{j\in \mathbb{Z}}\sum_{l\in \mathbb{Z}}\Psi\left(|\mathcal{P}_{\alpha}^{+}(|G|)(z_{l,j})|\right)2^{j\gamma (\alpha +2)} &\lesssim \sum_{j\in \mathbb{Z}}\sum_{l\in \mathbb{Z}} \int_{\mathcal{D}_{_{s_{\delta}}}\left(z_{l,j}\right)}  \Psi\left(|\mathcal{P}_{\alpha}^{+}(|G|)(z_{l,j})|\right)dV_{\alpha}(z) \\
&= \sum_{j\in \mathbb{Z}}\sum_{l\in \mathbb{Z}} \int_{\mathcal{D}_{_{s_{\delta}}}\left(z_{l,j}\right)}  \Psi\left(\int_{\mathbb{C}_{+}}|K_{\alpha}(z_{l,j},\omega)| |G(\omega)|dV_{\alpha}(\omega) \right)dV_{\alpha}(z),
\end{align*}
where $s_{\delta}$ is a constant in (\ref{eq:uaa5n}). Let $l,j\in \mathbb{Z}$. For $z \in \mathcal{D}_{_{s_{\delta}}}\left(z_{l,j}\right)$ and $\omega \in \mathbb{C}_{+}$, we have
\begin{align*}
|z-\overline{\omega}| \leq |z-z_{l,j}|+|z_{l,j}-\overline{\omega}|  &\Rightarrow \frac{|z-\overline{\omega}|}{|z_{l,j}-\overline{\omega}|}\leq  \frac{s_{\delta}\mathrm{Im}(z_{l,j})}{|z_{l,j}-\overline{\omega}|}+1 \leq s_{\delta} + 1 \leq 2 \\
&\Rightarrow \frac{|K_{\alpha}(z_{l,j},\omega)|}{|K_{\alpha}(z,\omega)|}  \leq 2^{\alpha+2}.
\end{align*}
Since $\{\mathcal{D}_{s_{\delta}}\left(z_{l,j}\right)\}_{l,j\in \mathbb{Z}}$  is a sequence of pairwise disjoint sets,   it follows that
\begin{align*}
&\sum_{j\in \mathbb{Z}}\sum_{l\in \mathbb{Z}}\int_{\mathcal{D}_{_{s_{\delta}}}\left(z_{l,j}\right)}  \Psi\left(\int_{\mathbb{C}_{+}}|K_{\alpha}(z_{l,j},\omega)| |G(\omega)|dV_{\alpha}(\omega) \right)dV_{\alpha}(z) \\
 &=\sum_{j\in \mathbb{Z}}\sum_{l\in \mathbb{Z}} \int_{\mathcal{D}_{_{s_{\delta}}}\left(z_{l,j}\right)}  \Psi\left(\int_{\mathbb{C}_{+}}|K_{\alpha}(z,\omega)|\frac{|K_{\alpha}(z_{l,j},\omega)|}{|K_{\alpha}(z,\omega)|}|G(\omega)|dV_{\alpha}(\omega)\right)dV_{\alpha}(z) \\
 &\lesssim\sum_{j\in \mathbb{Z}}\sum_{l\in \mathbb{Z}} \int_{\mathcal{D}_{_{s_{\delta}}}\left(z_{l,j}\right)}  \Psi\left(\int_{\mathbb{C}_{+}}|K_{\alpha}(z,\omega)||G(\omega)|dV_{\alpha}(\omega)\right)dV_{\alpha}(z) \\
 &= \int_{\cup_{j,l \in \mathbb{Z} }\mathcal{D}_{_{s_{\delta}}}\left(z_{l,j}\right)}  \Psi\left(|\mathcal{P}_{\alpha}^{+}(|G|)(z)\right)dV_{\alpha}(z) \\ 
 &\lesssim \int_{\mathbb{C}_{+}}  \Psi\left(|\mathcal{P}_{\alpha}^{+}(|G|)(z)\right)dV_{\alpha}(z)  \lesssim \int_{\mathbb{C}_{+}} \Psi(|G(z)|) dV_{\alpha}(z) \lesssim 1, 
  \end{align*}
where in the last line, we have used Theorem \ref{pro:mainplaqaaqaaqq6}. We deduce that the function $\sum_{j\in \mathbb{Z}}\sum_{l\in \mathbb{Z}}|F_{l,j}|$ belongs to $L^{\Phi}(\mathbb{C}_{+}, dV_{\alpha})$ and 
$$\left\|\sum_{j\in \mathbb{Z}}\sum_{l\in \mathbb{Z}}|F_{l,j}|\right\|_{L^{\Phi}_{\alpha}}^{lux} \lesssim \left\|\mu\right\|_{\ell^{\Phi}_{\alpha}}^{lux}. $$
Since $F_{l,j}$ and $\sum_{j\in \mathbb{Z}}\sum_{l\in \mathbb{Z}}F_{l,j}$ are analytic functions on  $\mathbb{C}_{+}$, it follows that they belong to  $A_{\alpha}^{\Phi}(\mathbb{C}_{+})$.

\medskip

Conversely,  
since  $F_{l,j}$ is analytic on $\mathbb{C}_{+}$, we have
$$ \Phi\left(|F_{l,j}(z_{l,j})|\right)\lesssim \frac{1}{\mathrm{Im}( z_{l,j})^{2+\alpha}} \int_{\mathcal{D}_{s_{\delta}}\left(z_{l,j}\right)}\Phi\left(|F_{l,j}(\omega)|\right)dV_{\alpha}(\omega),   $$
thanks to Proposition \ref{pro:main1q8}. Moreover, it is easy to verify that $  F_{l,j}(z_{l,j})= \mu_{l,j}.  $ Put
$  \lambda:=  \left\|\sum_{j\in \mathbb{Z}}\sum_{l\in \mathbb{Z}}|F_{l,j}|\right\|_{L^{\Phi}_{\alpha}}^{lux}.$
We have
\begin{align*}
\sum_{j\in \mathbb{Z}}\sum_{l\in \mathbb{Z}}\Phi\left(\frac{|\mu_{l,j}|}{\lambda}\right)2^{j\gamma (\alpha +2)} &\lesssim \sum_{j\in \mathbb{Z}}\sum_{l\in \mathbb{Z}} \int_{\mathcal{D}_{s_{\delta}}\left(z_{l,j}\right)}  \Phi\left(\frac{\sum_{j\in \mathbb{Z}}\sum_{l\in \mathbb{Z}}|F_{l,j}(\omega)|}{\lambda}\right)dV_{\alpha}(\omega) \\
&\lesssim \int_{\mathbb{C}_{+}}\Phi\left(\frac{\sum_{j\in \mathbb{Z}}\sum_{l\in \mathbb{Z}}|F_{l,j}(\omega)|}{\lambda}\right)dV_{\alpha}(\omega) \lesssim  1, 
\end{align*}
We  deduce that, $ \left\|\mu\right\|_{\ell^{\Phi}_{\alpha}}^{lux}\lesssim \lambda.$
\end{proof}

\begin{proposition}\label{pro:main12am0}
Let $\alpha>-1$,   $0<\delta<1$ and  $\Phi \in  \mathscr{U}\cap \nabla_{2}$. The sequence space $(\ell^{\Phi}_{\alpha}(\mathbb{C_{+}}),  \|.\|_{\ell^{\Phi}_{\alpha}}^{lux})$ is a Banach space.
\end{proposition}

\begin{proof}
Let $(\lambda^{(n)})_{n\in \mathbb{N}}$ be a Cauchy sequence of elements of $\ell^{\Phi}_{\alpha}(\mathbb{C_{+ }})$, where $\lambda^{(n)}:= \{\lambda_{l,j}^{(n)}\}_{l,j\in \mathbb{Z}}$, for all $n\in \mathbb{N}$.
For $\varepsilon >0$,  there exists $n_{\varepsilon}\in \mathbb{N}$ such that for all $ n,m \geq n_{\varepsilon}$, $\|\lambda^{(n)}-\lambda^{(m)}\|_{\ell^{\Phi}_{\alpha}}^{lux}\leq \varepsilon.$
For  $l,j\in \mathbb{Z}$ and $n\in \mathbb{N}$, 
consider the function $F_{l,j}^{(n)}$ defined by
$$  F_{l,j}^{(n)}(z) =2^{\alpha+2}\lambda_{l,j}^{(n)}\left(  \frac{z-\overline{z_{l,j}}}{i} \right)^{-\alpha-2}2^{j\gamma (\alpha +2)},~~ \forall~z \in \mathbb{C}_{+}.
   $$
By construction the function  $F_{l,j}^{(n)}$ belongs to  $A_{\alpha}^{\Phi}(\mathbb{C}_{+})$. Moreover,  $F_{l,j}^{(n)}(z_{l,j})= \lambda_{l,j}^{(n)}$ and 
$$  \left\|\sum_{j\in \mathbb{Z}}\sum_{l\in \mathbb{Z}}|F_{l,j}^{(n)}|\right\|_{L^{\Phi}_{\alpha}}^{lux} \sim \left\|\lambda^{(n)}\right\|_{\ell^{\Phi}_{\alpha}}^{lux},  $$
thanks to Lemma \ref{pro:mainqaqaqq6}. For  $ n,m \geq n_{\varepsilon}$,  we have 
 $$\left\|F_{l,j}^{(n)}-F_{l,j}^{(m)} \right\|_{A^{\Phi}_{\alpha}}^{lux} \lesssim \left\|\sum_{j\in \mathbb{Z}}\sum_{l\in \mathbb{Z}}|F_{l,j}^{(n)}-F_{l,j}^{(m)}|\right\|_{L^{\Phi}_{\alpha}}^{lux}\lesssim \left\|\lambda^{(n)}-\lambda^{(m)}\right\|_{\ell^{\Phi}_{\alpha}}^{lux}\lesssim \varepsilon.$$
We deduce that,   $\left(F_{l,j}^{(n)}\right)_{n}$ is a Cauchy sequence in $A^{\Phi}_{\alpha}(\mathbb{C_{+ }})$. Since $A^{\Phi}_{\alpha}(\mathbb{C_{+ }})$ is a Banach space, there exists $F \in A^{\Phi}_{\alpha}(\mathbb{C_{+ }})$ and  $n_{\varepsilon}'\in \mathbb{N}$ such that for  $ n \geq n_{\varepsilon}'$, 
$   \|F_{l,j}^{(n)}-F\|_{A^{\Phi}_{\alpha}}^{lux} \leq \varepsilon. $
For  $ n \geq n_{\varepsilon}'$, we have
\begin{align*}
\Phi\left( \frac{|\lambda_{l,j}^{(n)}-F(z_{l,j})|}{\varepsilon}  \right)2^{j\gamma (\alpha +2)}&=
\Phi\left( \frac{|F_{l,j}^{(n)}(z_{l,j})-F(z_{l,j})|}{\varepsilon}  \right)2^{j\gamma (\alpha +2)} \\
&\lesssim \int_{\mathcal{D}_{s_{\delta}}\left(z_{l,j}\right)}\Phi\left( \frac{\sum_{j\in \mathbb{Z}}\sum_{l\in \mathbb{Z}}|F_{l,j}^{(n)}(\omega)-F(\omega)|}{\varepsilon} \right)dV_{\alpha}(\omega),
\end{align*}
where $s_{\delta}$ is a constant in (\ref{eq:uaa5n}). It follows that 
\begin{align*}
\sum_{j\in \mathbb{Z}}\sum_{l\in \mathbb{Z}}\Phi\left( \frac{|\lambda_{l,j}^{(n)}-F_{l,j}(z_{l,j})|}{\varepsilon}  \right)2^{j\gamma (\alpha +2)}
&\lesssim \sum_{j\in \mathbb{Z}}\sum_{l\in \mathbb{Z}}\int_{\mathcal{D}_{s_{\delta}}\left(z_{l,j}\right)}\Phi\left( \frac{\sum_{j\in \mathbb{Z}}\sum_{l\in \mathbb{Z}}|F_{l,j}^{(n)}(\omega)-F(\omega)|}{\varepsilon} \right)dV_{\alpha}(\omega) \\
&\lesssim \int_{\mathbb{C}_{+}}\Phi\left( \frac{\sum_{j\in \mathbb{Z}}\sum_{l\in \mathbb{Z}}|F_{l,j}^{(n)}(\omega)-F(\omega)|}{\varepsilon} \right)dV_{\alpha}(\omega) \\
&\lesssim \liminf_{m\to \infty}\int_{\mathbb{C}_{+}}\Phi\left( \frac{\sum_{j\in \mathbb{Z}}\sum_{l\in \mathbb{Z}}|F_{l,j}^{(n)}(\omega)-F_{l,j}^{(m)}(\omega)|}{\varepsilon} \right)dV_{\alpha}(\omega) \\
&\lesssim 1,
\end{align*}
thanks to Fatou's Lemma. We deduce that 
$$  \left\|\lambda^{(n)}-\lambda\right\|_{\ell^{\Phi}_{\alpha}}^{lux}\lesssim \varepsilon,    $$
where $\lambda:= \{F(z_{l,j})\}_{l,j\in \mathbb{Z}}$. It follows that the sequence $\lambda$ belongs to $\ell^{\Phi}_{\alpha}(\mathbb{C_{+}})$. Indeed, for  $ n \geq n_{\varepsilon}'$, 
$$ \left\|\lambda\right\|_{\ell^{\Phi}_{\alpha}}^{lux} \lesssim \left\|\lambda^{(n)}\right\|_{\ell^{\Phi}_{\alpha}}^{lux}+  \left\|\lambda^{(n)}-\lambda\right\|_{\ell^{\Phi}_{\alpha}}^{lux}\lesssim \left\|\lambda^{(n)}\right\|_{\ell^{\Phi}_{\alpha}}^{lux}+ \varepsilon < \infty.    $$
We then conclude that the sequence space $(\ell^{\Phi}_{\alpha}(\mathbb{C_{+}}),  \|.\|_{\ell^{\Phi}_{\alpha}}^{lux})$ is a Banach space.
\end{proof}

Let   $\alpha>-1$,  $0<\delta<1$ and $\Phi \in \mathscr{U} \cap \nabla_{2}$. We introduce the following space
$$\mathcal{A}_\alpha^{\Phi}(\mathbb{C}_+):=\left\{F=2^{\alpha+2}\sum_{j\in \mathbb{Z}}\sum_{l\in \mathbb{Z}}\lambda_{l,j}K_\alpha(z,z_{l,j})2^{j\gamma (\alpha +2)},\textrm{for some}\quad \{\lambda_{l,j}\}_{l,j\in \mathbb{Z}}\in \ell_{\alpha}^{\Phi}(\mathbb{C}_{+})\right\}.$$
For $F\in \mathcal{A}_\alpha^{\Phi}(\mathbb{C}_+)$, we put $$\Vert F\Vert_{\mathcal{A}_\alpha^{\Phi}}:=\inf\Vert \{\lambda_{l,j}\}_{l,j\in \mathbb{Z}}\Vert_{\ell_\alpha^\Phi}^{lux} $$
where the infimum is taken over all sequences $\{\lambda_{l,j}\}_{l,j\in \mathbb{Z}}\in \ell_{\alpha}^{\Phi}(\mathbb{C}_{+})$ such that $$F=2^{\alpha+2}\sum_{j\in \mathbb{Z}}\sum_{l\in \mathbb{Z}}\lambda_{l,j}K_\alpha(z,z_{l,j})2^{j\gamma (\alpha +2)}.$$ We observe that endowed with the norm $\Vert \cdot \Vert_{\mathcal{A}_\alpha^{\Phi}}$, $\mathcal{A}_\alpha^{\Phi}(\mathbb{C}_+)$ is a Banach space.
\vskip .1cm
Given $F\in \mathcal{A}_\alpha^{\Phi}(\mathbb{C}_+$, we denote by $F_\lambda$ the representation of $F$ associated to the sequence $\lambda=\{\lambda_{l,j}\}_{l,j\in \mathbb{Z}}$. From Lemma \ref{pro:mainqaqaqq6}, we have that $\mathcal{A}_\alpha^{\Phi}(\mathbb{C}_+)$ embeds continuously into ${A}_\alpha^{\Phi}(\mathbb{C}_+)$. From the same Lemma, we have that if $F_\lambda$ and $F_\mu$ are two representations of $F$, then 
$$\Vert\{\mu_{l,j}\}_{l,j\in \mathbb{Z}}\Vert_{\ell_\alpha^\Phi}^{lux}\sim \Vert \{\lambda_{l,j}\}_{l,j\in \mathbb{Z}}\Vert_{\ell_\alpha^\Phi}^{lux}.$$  
\vskip .1cm
We have the following duality result.

\begin{theorem}\label{pro:main 5aqkqa3pl}
Let   $\alpha>-1$,  $0<\delta<1$ and $\Phi \in \mathscr{U} \cap \nabla_{2}$. The topological dual $\left(\ell_{\alpha}^{\Phi}(\mathbb{C}_{+})\right)^{*}$   of $\ell_{\alpha}^{\Phi}(\mathbb{C}_{+})$     is isomorphic to $\ell_{\alpha}^{\Psi}(\mathbb{C}_{+})$, in the sense that, for all  $T\in \left(\ell_{\alpha}^{\Phi}(\mathbb{C}_{+})\right)^{*}$, there is a unique $\mu:=\{\mu_{l,j}\}_{l,j \in \mathbb{Z} }\in \ell_{\alpha}^{\Psi}(\mathbb{C}_{+})$ such that for all $\lambda:=\{\lambda_{l,j}\}_{l,j \in \mathbb{Z} }\in \ell_{\alpha}^{\Phi}(\mathbb{C}_{+})$, 
 $$    T (\lambda)= c_{\alpha}\sum_{j\in \mathbb{Z}}\sum_{l\in \mathbb{Z}}\lambda_{l,j}\overline{\mu_{l,j}}2^{j\gamma (\alpha +2)} , 
   $$
where $\Psi$ is the complementary function of $\Phi$.
\end{theorem}

\begin{proof}
Let  $T\in \left(\ell_{\alpha}^{\Phi}(\mathbb{C}_{+})\right)^{*}$. Put 
$$   T_{1}(F_\lambda) = T\left(  \lambda\right), ~~ \forall~ \lambda:=\{\lambda_{l,j}\}_{l,j \in \mathbb{Z} } \in \ell_{\alpha}^{\Phi}(\mathbb{C}_{+}), $$
where $F_\lambda$ is the function defined by
$$ F_\lambda(z)=2^{\alpha+2}\sum_{j\in \mathbb{Z}}\sum_{l\in \mathbb{Z}}\lambda_{l,j}K_{\alpha}(z, z_{l,j})
2^{j\gamma (\alpha +2)},~~ \forall~z \in \mathbb{C}_{+}.   $$
By construction,  $T_{1}\in \left(\mathcal{A}_{\alpha}^{\Phi}(\mathbb{C}_{+})\right)^{*}$. Since  $\mathcal{A}_\alpha^{\Phi}(\mathbb{C}_+)$ embeds continuously into $A_\alpha^{\Phi}(\mathbb{C}_+)$,  then $T_{1}$ can be extended to an element of $\left(A_{\alpha}^{\Phi}(\mathbb{C}_+)\right)^{*}$  with the same operator norm. We still denote this extension $T_{1}$. Thus we have that there exists a unique $G\in A_{\alpha}^{\Psi}(\mathbb{C}_{+})$ such that for any $F\in  A_{\alpha}^{\Phi}(\mathbb{C}_{+})$,
 $$T_{1}(F)=\int_{\mathbb{C}_+}F(z)\overline{G(z)}dV_\alpha(z).$$
In particular, for any $F\in \mathcal{A}_\alpha^{\Phi}(\mathbb{C}_+)$, with $$F=F_{\lambda}(z) =2^{\alpha+2}\sum_{j\in \mathbb{Z}}\sum_{l\in \mathbb{Z}}\lambda_{l,j}K_{\alpha}(z, z_{l,j})
2^{j\gamma (\alpha +2)},~~ \forall~z \in \mathbb{C}_{+},$$
 since $P_\alpha$ reproduces functions in $A_{\alpha}^{\Psi}(\mathbb{C}_{+})$ (see Lemma \ref{pro:mainplaqzepama6}), we obtain
\begin{align*}
T_{1}(F) &:= \int_{\mathbb{C}_{+}} F_{\lambda}(z)\overline{G(z)}dV_{\alpha}(z)\\  
&=2^{\alpha+2}\sum_{j\in \mathbb{Z}}\sum_{l\in \mathbb{Z}}\lambda_{l,j}2^{j\gamma (\alpha +2)}\int_{\mathbb{C}_{+}} \overline{K_{\alpha}(z_{l,j},z)G(z)}dV_{\alpha}(z)   \\
&=2^{\alpha+2}\sum_{j\in \mathbb{Z}}\sum_{l\in \mathbb{Z}}\lambda_{l,j}\overline{\mathcal{P}_{\alpha}(G)(z_{l,j})}2^{j\gamma (\alpha +2)} 
= 2^{\alpha+2}\sum_{j\in \mathbb{Z}}\sum_{l\in \mathbb{Z}}\lambda_{l,j}\overline{G(z_{l,j})}2^{j\gamma (\alpha +2)}. 
\end{align*}
It follows that, for all $\lambda:=\{\lambda_{l,j}\}_{l,j \in \mathbb{Z} } \in \ell_{\alpha}^{\Phi}(\mathbb{C}_{+})$, we have
$$  T(\lambda)=T_1(F_\lambda)= 2^{\alpha+2}\sum_{j\in \mathbb{Z}}\sum_{l\in \mathbb{Z}}\lambda_{l,j}\overline{\mu_{l,j}}2^{j\gamma (\alpha +2)}, 
   $$
where $\mu_{l,j}:=G(z_{l,j})$,  for all $l,j\in\mathbb{Z}$.   
\end{proof}

\subsection{Averaging functions and Berezin transform on $\mathbb{C}_{+}$.}
  For $\alpha > -1$ and $\omega \in \mathbb{C}_{+}$, the normalized reproducing kernel at $\omega$ is given by
   \begin{equation}\label{eq:espqaqeqs1}
   k_{\alpha,\omega}(z)= \frac{\mathrm{Im}( \omega)^{(2+\alpha)/2}}{(z-\overline{\omega})^{2+\alpha}}, ~~\forall~z \in \mathbb{C_{+}}.
    \end{equation} 
   It is easy to verify that
   \begin{equation}\label{eq:espqaqs1}
   |k_{\alpha,\omega}(z)|^{2} \leq |K_{\alpha}(z,\omega)| , ~~\forall~z, \omega\in \mathbb{C_{+}},
    \end{equation} 
   where $K_{\alpha}$ is the function  defined in (\ref{eq:fo4nAQ1l}). 
    
 \medskip
  
Let $\mu$ be a positive measure on $\mathbb{C}_{+}$ and  $0<s< 1$. The average function  $\widehat{\mu}_{s}$ of  $\mu$ is the function defined by
 $$  \widehat{\mu}_{s}(z)= \frac{\mu(\mathcal{D}_{s}(z))}{| \mathcal{D}_{s}(z)|_{\alpha}},~~ \forall~z \in \mathbb{C}_{+}.  $$
We recall that the Berezin transform $\widetilde{\mu}$ of the measure $\mu$ is the function defined for any
 $\omega \in \mathbb{C}_{+}$ by
$$ \widetilde{\mu}(\omega) = \int_{\mathbb{C}_{+}}|k_{\alpha,\omega}(z)|^{2}d\mu(z).
 $$

 We have the following.
 \begin{proposition}\label{pro:main1qaqaqp8}
 Let $\alpha > -1$,   $\mu$ a positive measure on $\mathbb{C}_{+}$ and $z \in \mathbb{C}_{+}$. For  $0<s < 1$, there exists a constant $C_{1}:=C_{s,\alpha}> 0$ such that 
 \begin{equation}\label{eq:espqepmqs1}
 \widehat{\mu}_{s}(z) \leq C_{1}\widetilde{\mu}(z).
 \end{equation} 
 Conversely, if  $0<s < 1/2$ and   $s':=\frac{s}{1-s}$, then there exists a constant $C_{2}:=C_{s,\alpha}> 0$ such that 
 \begin{equation}\label{eq:espqeaqpmqs1}
 \widetilde{\mu}(z) \leq C_{2}\mathcal{P}_{\alpha}^{+}(\widehat{\mu}_{s'})(z).
 \end{equation} 
 \end{proposition}
 
 \begin{proof}
 For $z\in \mathbb{C_{+}}$, we have: 
 \begin{align*}
 \widehat{\mu}_{s}(z)&=\frac{\mu(\mathcal{D}_{s}(z))}{| \mathcal{D}_{s}(z)|_{\alpha}}
 =\int_{\mathbb{C_{+}}}\frac{| \mathcal{D}_{s}(z)|_{\alpha}}{\left(| \mathcal{D}_{s}(z)|_{\alpha}\right)^{2}}\chi_{\mathcal{D}_{s}(z)}(\omega) d\mu(\omega)\\
 &\sim\int_{\mathbb{C_{+}}}\frac{\mathrm{Im}( \omega)^{2+\alpha}}{|\omega-\overline{z}|^{2(2+\alpha)}}\chi_{\mathcal{D}_{s}(z)}(\omega) d\mu(\omega) \\
 &\lesssim \int_{\mathbb{C_{+}}}\frac{\mathrm{Im}( \omega)^{2+\alpha}}{|\omega-\overline{z}|^{2(2+\alpha)}} d\mu(\omega) 
 =  \widetilde{\mu}(z).
 \end{align*}
 
 Conversely, suppose that $0<s < 1/2$ and put  $s':=\frac{s}{1-s}$. We have 
 \begin{align*}
 \widetilde{\mu}(z)
 &\lesssim  \int_{\mathbb{C}_{+}} \left(  \frac{1}{\mathrm{Im}( \omega)^{2+\alpha}}\int \int_{\mathcal{D}_{s}\left(\omega\right)}|k_{\alpha,z}(\xi)|^{2}dV_{\alpha}(\xi)  \right)d\mu(\omega) \\
 &=  \int_{\mathbb{C}_{+}} \left( \int_{\mathbb{C}_{+}} \frac{1}{\mathrm{Im}( \omega)^{2+\alpha}}|k_{\alpha,z}(\xi)|^{2}\chi_{\mathcal{D}_{s}\left(\omega\right)}(\xi)dV_{\alpha}(\xi)  \right)d\mu(\omega) \\
 &\lesssim  \int_{\mathbb{C}_{+}} |k_{\alpha,z}(\xi)|^{2}\left(\frac{1}{\mathrm{Im}( \xi)^{2+\alpha}} \int_{\mathbb{C}_{+}}  \chi_{\mathcal{D}_{s'}\left(\xi\right)}( \omega) d\mu(\omega)\right)dV_{\alpha}(\xi) \\
 &\lesssim  \int_{\mathbb{C}_{+}} |k_{\alpha,z}(\xi)|^{2}\widehat{\mu}_{s'}(\xi)dV_{\alpha}(\xi)
 \lesssim  \int_{\mathbb{C}_{+}} |K_{\alpha}(\xi,z)|\widehat{\mu}_{s'}(\xi)dV_{\alpha}(\xi)=\mathcal{P}_{\alpha}^{+}(\widehat{\mu}_{s'})(z),
 \end{align*}
 thanks to Fubbini's Theorem, relation (\ref{eq:iaqnaqeapq}) and  inequality (\ref{eq:espqaqs1}). 
 \end{proof}

The following result is an immediate consequence of Proposition \ref{pro:main1qaqaqp8} and Theorem \ref{pro:mainplaqaaqaaqq6}. Therefore, the proof will be omitted.

\begin{lemma}\label{pro:main1qppaq8}
Let $\alpha > -1$,   $\Phi \in \mathscr{U}\cap \nabla_{2}$ and  $\mu$ a positive measure on $\mathbb{C}_{+}$. Let $0<s < 1/2$ and put  $s':=\frac{s}{1-s}$.   The following assertions are equivalent: 
\begin{itemize}
\item[(i)] $\widetilde{\mu} \in L^{\Phi}(\mathbb{C_{+}}, dV_{\alpha})$; 
\item[(ii)]  $\widehat{\mu}_{s'} \in L^{\Phi}(\mathbb{C_{+}}, dV_{\alpha})$.
\end{itemize}
Moreover,
\begin{equation}\label{eq:inmqqay}
 \|\widehat{\mu}_{s'}\|_{L^{\Phi}_{\alpha}}^{lux} \sim \|\widetilde{\mu}\|_{L^{\Phi}_{\alpha}}^{lux}. 
 \end{equation} 
\end{lemma}

The following lemma establishes the connection between the continuous and discrete Orlicz conditions associated with the local averages of a positive measure. It provides both directions of the correspondence under suitable restrictions on the parameter $\delta$.

\begin{lemma}\label{pro:main1qpapqp8}
Let  $0<\delta < 1$, $\alpha > -1$,  $\Phi \in \mathscr{L}\cup \mathscr{U}$ and  $\mu$ a positive measure on $\mathbb{C}_{+}$. The following assertions are satisfied.
\begin{itemize}
\item[(i)] Suppose that  $0<\delta < 1/3$ and put $\delta_{1}:=\frac{2\delta}{1-\delta}$. If
  $\widehat{\mu}_{\delta_{1}} \in L^{\Phi}(\mathbb{C_{+}}, dV_{\alpha})$, then   $\{ \widehat{\mu}_{\delta}(z_{l,j}) \}_{l,j \in \mathbb{Z}} \in \ell_{\alpha}^{\Phi}(\mathbb{C}_{+})$.  Moreover, there exists $C_{1}:=C_{\alpha,\delta}>0$ a constant such that
\begin{equation}\label{eq:inaaqay}
\left\|\{ \widehat{\mu}_{\delta}(z_{l,j}) \}\right\|_{\ell^{\Phi}_{\alpha}}^{lux}\leq C_{1}\|\widehat{\mu}_{\delta_{1}}\|_{L^{\Phi}_{\alpha}}^{lux}.
 \end{equation}
\item[(ii)] Suppose that  $0<\delta < -1+ \sqrt{2}$ and put $\delta_{2}:=\delta(2+\delta)$. If $\{ \widehat{\mu}_{\delta_{2}}(z_{l,j}) \}_{l,j \in \mathbb{Z}} \in \ell_{\alpha}^{\Phi}(\mathbb{C}_{+})$, then $\widehat{\mu}_{\delta} \in L^{\Phi}(\mathbb{C_{+}}, dV_{\alpha})$.  Moreover, there exists $C_{2}:=C_{\alpha,\delta}>0$ a constant such that
\begin{equation}\label{eq:inaaqaqay}
\|\widehat{\mu}_{\delta}\|_{L^{\Phi}_{\alpha}}^{lux}\leq C_{2}\left\|\{ \widehat{\mu}_{\delta_{2}}(z_{l,j}) \}\right\|_{\ell^{\Phi}_{\alpha}}^{lux}.
 \end{equation}
\end{itemize}
\end{lemma}

\begin{proof}
Let  $l,j \in \mathbb{Z}$ and $z \in \mathcal{D}_{\delta} \left(z_{l,j}\right)$.  It is readily known that
$$  \mathcal{D}_{\delta} \left(z_{l,j}\right) \subset  \mathcal{D}_{\delta_{1}} \left(z\right) \hspace*{0.5cm},\hspace*{0.5cm} |\mathcal{D}_{\delta}(z_{l,j})|_{\alpha} \sim |\mathcal{D}_{\delta_{1}}\left(z\right)|_{\alpha}, $$
and 
$$  \mathcal{D}_{\delta} \left(z\right) \subset  \mathcal{D}_{\delta_{2}} \left(z_{l,j}\right) \hspace*{0.5cm}\textrm{and} \hspace*{0.5cm} |\mathcal{D}_{\delta}(z)|_{\alpha} \sim |\mathcal{D}_{\delta_{2}}\left(z_{l,j}\right)|_{\alpha}. $$
It follows that
$$  \widehat{\mu}_{\delta}(z_{l,j})= \frac{\mu(\mathcal{D}_{\delta}(z_{l,j}))}{| \mathcal{D}_{\delta}(z_{l,j})|_{\alpha}} \lesssim  \frac{\mu(\mathcal{D}_{\delta_{1}}\left(z\right))}{| \mathcal{D}_{\delta_{1}}\left(z\right)|_{\alpha}} =\widehat{\mu}_{\delta_{1}}(z)
  $$
and 
$$  \widehat{\mu}_{\delta}(z)= \frac{\mu(\mathcal{D}_{\delta}(z))}{| \mathcal{D}_{\delta}(z)|_{\alpha}} \lesssim  \frac{\mu(\mathcal{D}_{\delta_{1}}\left(z_{l,j}\right))}{| \mathcal{D}_{\delta_{2}}\left(z_{l,j}\right)|_{\alpha}} =\widehat{\mu}_{\delta_{2}}(z_{l,j}).
  $$

\medskip

$i)$ Suppose that  $\widehat{\mu}_{\delta_{1}} \in L^{\Phi}(\mathbb{C_{+}}, dV_{\alpha})$. We prove that $\{ \widehat{\mu}_{\delta}(z_{l,j}) \}_{l,j \in \mathbb{Z}} \in \ell_{\alpha}^{\Phi}(\mathbb{C}_{+})$. 
We can assume without loss of generality that  $\mu \not\equiv 0$ and $\|\widehat{\mu}_{\delta_{1}}\|_{L^{\Phi}_{\alpha}}^{lux}=1$.

Since the disks  $\mathcal{D}_{s_{\delta}} \left(z_{l,j}\right)$ are pairwise disjoint, and each $\mathcal{D}_{s_{\delta}} \left(z_{l,j}\right)$ is contained in  the disk $\mathcal{D}_{\delta} \left(z_{l,j}\right)$, where $s_{\delta}$ is the constant defined in the relation (\ref{eq:uaa5n}), we have
\begin{align*}
\sum_{j \in \mathbb{Z}}\sum_{l \in \mathbb{Z}}  \Phi\left( |\widehat{\mu}_{\delta}(z_{l,j})|\right)2^{j\gamma (\alpha +2)} &\sim  \sum_{j \in \mathbb{Z}}\sum_{l \in \mathbb{Z}} \int_{\mathcal{D}_{s_{\delta}} \left(z_{l,j}\right)} \Phi\left( |\widehat{\mu}_{\delta}(z_{l,j})| \right) dV_{\alpha}(z) \\
&\lesssim  \sum_{j \in \mathbb{Z}}\sum_{l \in \mathbb{Z}} \int_{\mathcal{D}_{s_{\delta}} \left(z_{l,j}\right)} \Phi\left( |\widehat{\mu}_{\delta_{1}}(z)|\right) dV_{\alpha}(z) \\ 
&\lesssim   \int_{\mathbb{C}_{+}} \Phi\left( |\widehat{\mu}_{\delta_{1}}(z)| \right) dV_{\alpha}(z) 
\lesssim  1.
\end{align*}

\medskip

$ii)$  Suppose that   $\{ \widehat{\mu}_{\delta_{2}}(z_{l,j}) \}_{l,j \in \mathbb{Z}} \in \ell_{\alpha}^{\Phi}(\mathbb{C}_{+})$, then  $\widehat{\mu}_{\delta} \in L^{\Phi}(\mathbb{C_{+}}, dV_{\alpha})$. We can also assume without loss of generality that  $\mu \not\equiv 0$ and $\left\|\{ \widehat{\mu}_{\delta_{2}}(z_{l,j}) \}\right\|_{\ell^{\Phi}_{\alpha}}^{lux}=1$. We have
\begin{align*}
\int_{\mathbb{C}_{+}}\Phi\left(|\widehat{\mu}_{\delta}(z)|\right)dV_{\alpha}(z)
&\lesssim \sum_{j \in \mathbb{Z}}\sum_{l \in \mathbb{Z}} \int_{\mathcal{D}_{\delta} \left(z_{l,j}\right)}  \Phi\left(|\widehat{\mu}_{\delta}(z)|\right)dV_{\alpha}(z) \\
&\lesssim \sum_{j \in \mathbb{Z}}\sum_{l \in \mathbb{Z}} \int_{\mathcal{D}_{\delta} \left(z_{l,j}\right)}  \Phi\left( |\widehat{\mu}_{\delta_{2}}(z_{l,j})| \right)dV_{\alpha}(z) \\
&\sim \sum_{j \in \mathbb{Z}}\sum_{l \in \mathbb{Z}}\Phi\left( |\widehat{\mu}_{\delta_{2}}(z_{l,j})| \right)2^{j\gamma (\alpha +2)}  
\lesssim 1.
\end{align*}
The proof is complete.
\end{proof}

 \section{Proof of main results.}

 \subsection{Proof of Theorem \ref{pro:mainqaaq6}.}
 
 \begin{lemma}\label{pro:mainqaaqaqp6}
Let $\alpha>-1$, $0<\delta<1$, and
$\Phi\in\mathscr{U}\cap\nabla_2$. For $\delta$ sufficiently
small and every $F\in A_\alpha^\Phi(\mathbb{C}_+)$, there exists
a sequence
$\mu=\{\mu_{l,j}\}_{l,j\in\mathbb{Z}}
\in\ell_\alpha^\Phi(\mathbb{C}_+)$ such that
\[
F(z)=c_{\alpha}
\sum_{j\in\mathbb{Z}}\sum_{l\in\mathbb{Z}}
\mu_{l,j}
\left(\frac{z-\overline{z_{l,j}}}{i}\right)^{-\alpha-2}
2^{j\gamma(\alpha+2)},
\qquad z\in\mathbb{C}_+.
\]
Moreover, there exist constants $C_1,C_2>0$, depending only
on $\Phi$, $\delta$, and $\alpha$, such that
\begin{equation}\label{eq:inaqay}
C_1\|\mu\|_{\ell_\alpha^\Phi}^{\mathrm{lux}}
\leq
\|F\|_{A_\alpha^\Phi}^{\mathrm{lux}}
\leq
C_2\|\mu\|_{\ell_\alpha^\Phi}^{\mathrm{lux}}.
\end{equation}
\end{lemma}

\begin{proof}
We first establish the existence of the required representation.
Define the linear operator
\[
T:\ell_\alpha^\Phi(\mathbb{C}_+)
\longrightarrow A_\alpha^\Phi(\mathbb{C}_+)
\]
by
\[
T(\lambda)(z)
=
2^{\alpha+2}
\sum_{j\in\mathbb{Z}}\sum_{l\in\mathbb{Z}}
\lambda_{l,j}
\left(\frac{z-\overline{z_{l,j}}}{i}\right)^{-\alpha-2}
2^{j\gamma(\alpha+2)},
\qquad z\in\mathbb{C}_+.
\]
By Lemma~\ref{pro:mainqaqaqq6}, the operator $T$ is well defined
and bounded from $\ell_\alpha^\Phi(\mathbb{C}_+)$ into
$A_\alpha^\Phi(\mathbb{C}_+)$.

It remains to prove that $T$ is surjective. To this end, consider
the operator
\[
T^*:A_\alpha^\Psi(\mathbb{C}_+)
\longrightarrow\ell_\alpha^\Psi(\mathbb{C}_+),
\qquad
T^*(G)=\{G(z_{l,j})\}_{l,j\in\mathbb{Z}},
\]
where $\Psi$ denotes the complementary function of $\Phi$.
By Lemma~\ref{pro:mainqaqa6}, this operator is well defined and
bounded.

Using the duality pairings on the corresponding Bergman--Orlicz
and sequence spaces, together with the reproducing property of
the Bergman kernel, we obtain
\begin{align}\label{eq:ineqqqay}
\langle T(\lambda),G\rangle_{A_\alpha^\Phi,A_\alpha^\Psi}
&=
2^{\alpha+2}
\sum_{j\in\mathbb{Z}}\sum_{l\in\mathbb{Z}}
\lambda_{l,j}\overline{G(z_{l,j})}
2^{j\gamma(\alpha+2)}
\nonumber\\
&=
\langle\lambda,T^*(G)\rangle_
{\ell_\alpha^\Phi,\ell_\alpha^\Psi}.
\end{align}
Thus, with these duality conventions, $T^*$ is the adjoint
of $T$.

Since \(T^*\) is the adjoint of \(T\), to establish the surjectivity of
\[
T:\ell_\alpha^\Phi(\mathbb{C}_+)\longrightarrow A_\alpha^\Phi(\mathbb{C}_+),
\]
it suffices to show that there exists a constant \(C>0\) such that
\[
\|G\|_{A_\alpha^\Psi}^{\mathrm{lux}}
\leq C\|T^*(G)\|_{\ell_\alpha^\Psi}^{\mathrm{lux}},
\qquad
\forall\,G\in A_\alpha^\Psi(\mathbb{C}_+),
\]
by the surjectivity criterion for adjoint operators (see Theorem II.19 in \cite{brezis}).

Let \(G\in A_\alpha^\Psi(\mathbb{C}_+)\). By Lemma \(\ref{pro:mainqaqa6}\), we obtain
\[
\begin{aligned}
\|G\|_{A_\alpha^\Psi}^{\mathrm{lux}}
&\lesssim
\|\{G(z_{l,j})\}_{l,j\in\mathbb{Z}}\|_{\ell_\alpha^\Psi}^{\mathrm{lux}}\\
&=\|T^*(G)\|_{\ell_\alpha^\Psi}^{\mathrm{lux}}.
\end{aligned}
\]
Therefore, \(T^*\) is bounded below. It follows from the aforementioned criterion that \(T\) is surjective from \(\ell_\alpha^\Phi(\mathbb{C}_+)\) onto \(A_\alpha^\Phi(\mathbb{C}_+)\). Hence, for every $F\in A_\alpha^\Phi(\mathbb{C}_+)$, there exists
$\mu\in\ell_\alpha^\Phi(\mathbb{C}_+)$ such that
\[
T(\mu)=F.
\]
By the definition of $T$, this gives the desired representation.

We now prove the norm estimates in \eqref{eq:inaqay}.

The inequality
\[
\|F\|_{A_\alpha^\Phi}^{\mathrm{lux}}
\lesssim
\|\mu\|_{\ell_\alpha^\Phi}^{\mathrm{lux}}
\]
follows directly from Lemma~\ref{pro:mainqaqaqq6}.

It remains to establish the reverse inequality. Let
$\lambda=\{\lambda_{l,j}\}_{l,j\in\mathbb{Z}}
\in\ell_\alpha^\Psi(\mathbb{C}_+)$ satisfy $\|\lambda\|_{\ell_\alpha^\Psi}^{\mathrm{lux}}\leq 1$.

By definition, the duality pairing is given by
\[
\langle\mu,\lambda\rangle_\alpha
=
c_\alpha
\sum_{j\in\mathbb{Z}}\sum_{l\in\mathbb{Z}}
\mu_{l,j}\overline{\lambda_{l,j}}
2^{j\gamma(\alpha+2)}.
\]

For each $l,j\in\mathbb{Z}$, define
\[
G_{l,j}(z)
=
2^{\alpha+2}\lambda_{l,j}
\left(\frac{z-\overline{z_{l,j}}}{i}\right)^{-\alpha-2}
2^{j\gamma(\alpha+2)},
\qquad z\in\mathbb{C}_+.
\]
By Lemma~\ref{pro:mainqaqaqq6}, the functions $G_{l,j}$ and
their sum belong to $A_\alpha^\Psi(\mathbb{C}_+)$.
Moreover, we have
\[
\left\|
\sum_{j\in\mathbb{Z}}\sum_{l\in\mathbb{Z}}
|G_{l,j}|
\right\|_{L_\alpha^\Psi}^{\mathrm{lux}}
\sim
\|\lambda\|_{\ell_\alpha^\Psi}^{\mathrm{lux}}.
\]

Since $G_{l,j}\in A_\alpha^\Psi(\mathbb{C}_+)$, the reproducing
property of the Bergman projection, given by
Lemma~\ref{pro:mainplaqzepama6}, yields
\[
G_{l,j}(z_{l,j})
=
\mathcal{P}_\alpha(G_{l,j})(z_{l,j}).
\]

Using this identity and the definition of the duality pairing,
we obtain
\begin{align*}
|\langle\overline{\mu},\lambda\rangle_\alpha|
&=
\left|
c_\alpha
\sum_{j\in\mathbb{Z}}\sum_{l\in\mathbb{Z}}
\overline{\mu_{l,j}}\,
\overline{\lambda_{l,j}}\,
2^{j\gamma(\alpha+2)}
\right|\\
&=
\left|
c_\alpha
\sum_{j\in\mathbb{Z}}\sum_{l\in\mathbb{Z}}
\overline{\mu_{l,j}}\,
\overline{\mathcal{P}_\alpha(G_{l,j})(z_{l,j})}\,
2^{j\gamma(\alpha+2)}
\right|.
\end{align*}
Substituting the integral representation of $\mathcal{P}_\alpha$
and interchanging the sum and the integral, we obtain
\begin{align*}
|\langle\overline{\mu},\lambda\rangle_\alpha|
&\lesssim
\int_{\mathbb{C}_+}
\left|
\sum_{j\in\mathbb{Z}}\sum_{l\in\mathbb{Z}}
\mu_{l,j}
\left(\frac{z-\overline{z_{l,j}}}{i}\right)^{-\alpha-2}
2^{j\gamma(\alpha+2)}
\right|
\times
\left|
\sum_{j\in\mathbb{Z}}\sum_{l\in\mathbb{Z}}
|G_{l,j}(z)|
\right|\,dV_\alpha(z)\\
&\lesssim
\int_{\mathbb{C}_+}
|F(z)|
\sum_{j\in\mathbb{Z}}\sum_{l\in\mathbb{Z}}
|G_{l,j}(z)|\,dV_\alpha(z).
\end{align*}

By the generalized Hölder inequality for complementary Orlicz
functions, it follows that
\begin{align*}
|\langle\overline{\mu},\lambda\rangle_\alpha|
&\lesssim
\|F\|_{A_\alpha^\Phi}^{\mathrm{lux}}
\left\|
\sum_{j\in\mathbb{Z}}\sum_{l\in\mathbb{Z}}
|G_{l,j}|
\right\|_{L_\alpha^\Psi}^{\mathrm{lux}}\\
&\lesssim
\|F\|_{A_\alpha^\Phi}^{\mathrm{lux}}
\|\lambda\|_{\ell_\alpha^\Psi}^{\mathrm{lux}}\\
&\lesssim
\|F\|_{A_\alpha^\Phi}^{\mathrm{lux}}.
\end{align*}

Taking the supremum over all
$\lambda\in\ell_\alpha^\Psi(\mathbb{C}_+)$ satisfying
$\|\lambda\|_{\ell_\alpha^\Psi}^{\mathrm{lux}}\leq 1$, and using
the duality between the corresponding Orlicz sequence spaces,
we conclude that
\[
\|\overline{\mu}\|_{\ell_\alpha^\Phi}^{\mathrm{lux}}
\lesssim
\|F\|_{A_\alpha^\Phi}^{\mathrm{lux}}.
\]
Finally, since complex conjugation preserves the Luxemburg norm,
we have
\[
\|\mu\|_{\ell_\alpha^\Phi}^{\mathrm{lux}}
=
\|\overline{\mu}\|_{\ell_\alpha^\Phi}^{\mathrm{lux}}.
\]
Therefore,
\[
\|\mu\|_{\ell_\alpha^\Phi}^{\mathrm{lux}}
\lesssim
\|F\|_{A_\alpha^\Phi}^{\mathrm{lux}},
\]
which completes the proof.
\end{proof}

 \proof[Proof of Theorem \ref{pro:mainqaaq6}.]
 The proof follows from Lemma \ref{pro:mainqaqaqq6} and Lemma \ref{pro:mainqaaqaqp6}.
 \epf

 \subsection{Proof of Theorem \ref{pro:mainqaaqaq6}.}

 We recall that the Rademacher functions
 $r_{n}$ are defined for $t\in [0,t]$ as follows
 $$   r_{n}(t)= \textit{sgn}[\sin (2^n\pi t)], ~~n=1,2,3,\cdots.  $$

 The following is obtained using Khintchine's double inequality (see \cite{defklaus}) and the ideas in \cite{sehba4}
 \begin{lemma}\label{pro:main2qapAaa}
 For  $\Phi \in \mathscr{L}\cup \mathscr{U}$, there
 exist constants $0 < A_{\Phi} \leq B_{\Phi} < \infty$ such that, for any sequence $\{x_{k,j}\} \in \ell^{2}$,  we have
 \begin{equation}\label{eqoaqbpmaq6}
 A_{\Phi}\left( \sum_{k,j}|x_{k,j}|^{2}   \right)^{1/2} \leq \Phi^{-1}\left(\int_0^1\int_{0}^{1}\Phi\left(\left|  \sum_{k,j}x_{k,j}r_{k}(t)r_j(s)  \right|\right)dtds \right)\leq B_{\Phi}\left( \sum_{k,j}|x_{k,j}|^{2}   \right)^{1/2}. \end{equation}
 \end{lemma}

  \proof[Proof of Theorem \ref{pro:mainqaaqaq6}.]

 Let us prove the assertion $(ii) \Rightarrow (i)$. 
 
 \medskip
  
 Let $F\in A_{\alpha}^{\Phi_{1}}(\mathbb{C}_{+})$. Assume that $F\not=0$, because there is nothing to show when $F=0$. We can also assume without loss of generality that $\|F\|_{A^{\Phi_{1}}_{\alpha}}^{lux}=1$.
 Let us choose $0<\delta<1/2$   and put $\delta':=\frac{\delta}{1-\delta}$.  
 For $z \in \mathbb{C}_{+}$, we have 
 \begin{align*}
 \Phi_{2}\left(|F(z)|\right)
 &\lesssim \frac{1}{\mathrm{Im}( z)^{2+\alpha}} \int_{\mathcal{D}_{\delta}\left(z\right)}\Phi_{2}\left(|F(\omega)|\right)dV_{\alpha}(\omega) \\
 &= \int_{\mathbb{C}_{+}} \Phi_{2}\left(|F(\omega)|\right)\chi_{\mathcal{D}_{\delta}\left(z\right)}(\omega)\frac{dV_{\alpha}(\omega)}{\mathrm{Im}( z)^{2+\alpha}} \\
 &\sim \int_{\mathbb{C}_{+}} \Phi_{2}\left(|F(\omega)|\right)\frac{\chi_{\mathcal{D}_{\delta'}\left(\omega\right)}(z)}{| \mathcal{D}_{\delta'}(\omega)|_{\alpha}}dV_{\alpha}(\omega),
 \end{align*}
 thanks to Proposition \ref{pro:main1q8} and,  relations (\ref{eq:espeqs1}), (\ref{eq:eqas1}) and (\ref{eq:iaqnaqeapq}). It follows that
 \begin{align*}
 \int_{\mathbb{C}_{+}}\Phi_{2}\left( |F(z) |\right)d\mu(z)&\lesssim \int_{\mathbb{C}_{+}} \left( \int_{\mathbb{C}_{+}} \Phi_{2}\left(|F(\omega)|\right)\frac{\chi_{\mathcal{D}_{\delta'}\left(\omega\right)}(z)}{| \mathcal{D}_{\delta'}(\omega)|_{\alpha}}dV_{\alpha}(\omega)  \right) d\mu(z) \\ 
 &= \int_{\mathbb{C}_{+}}\Phi_{2}\left(|F(\omega)|\right) \left(\frac{1}{| \mathcal{D}_{\delta'}(\omega)|_{\alpha}}  \int_{\mathbb{C}_{+}}\chi_{\mathcal{D}_{\delta'}\left(\omega\right)}(z)  d\mu(z) \right)dV_{\alpha}(\omega) \\
 &= \int_{\mathbb{C}_{+}}\Phi_{2}\left(|F(\omega)|\right) \widehat{\mu}_{\delta'}(\omega) dV_{\alpha}(\omega)\\
 &\lesssim \int_{\mathbb{C}_{+}}\Phi_{2}\left(|F(\omega)|\right) \widetilde{\mu}(\omega) dV_{\alpha}(\omega),
 \end{align*}
 thanks to Proposition \ref{pro:main1qaqaqp8}. 
 Since  $\widetilde{\mu}$ belongs to $L^{\Phi_{3}}(\mathbb{C}_{+}, dV_{\alpha})$, we have
 \begin{align*}
 \int_{\mathbb{C}_{+}}\Phi_{2}\left( |F(z) |\right)d\mu(z)&\lesssim \int_{\mathbb{C}_{+}}\Phi_{2}\left(|F(\omega)|\right)\widetilde{\mu}(\omega) dV_{\alpha}(\omega) \\ 
 &\lesssim \int_{\mathbb{C}_{+}}\Phi_{1}\circ\Phi_{2}^{-1}\left(\Phi_{2}\left(|F(\omega)|\right) \right)dV_{\alpha}(\omega)+\int_{\mathbb{C}_{+}}\Phi_{3}\left(\widetilde{\mu}(\omega) \right) dV_{\alpha}(\omega)   \\ 
 &= \int_{\mathbb{C}_{+}}\Phi_{1}\left(|F(\omega)|\right) dV_{\alpha}(\omega)+\int_{\mathbb{C}_{+}}\Phi_{3}\left(\widetilde{\mu}(\omega)\right) dV_{\alpha}(\omega)   
 \lesssim 1. 
 \end{align*}

 We are left to prove $(i) \Rightarrow (ii)$. 
 
 \medskip

 Let  $0<\delta < 1$ and $\nu := \{\nu_{l,j}\}_{l,j\in \mathbb{Z}} \in \ell_{\alpha}^{\Phi_{1}\circ\Phi_{2}^{-1}}(\mathbb{C}_{+})$. Suppose that $\nu \not=0$, because there is nothing to show when $\nu =0$. We can also without loss of generality assume that $\|\nu\|_{\ell_{\alpha}^{\Phi_{1}\circ\Phi_{2}^{-1}}}^{lux}=1$. For $l,j \in \mathbb{Z}$, put 
 $$  \lambda_{l,j}:= \Phi_{2}^{-1}\left( |\nu_{l,j}|\right).    $$
 By construction, the sequence $\lambda := \{\lambda_{l,j}\}_{l,j\in \mathbb{Z}} \in \ell_{\alpha}^{\Phi_{1}}(\mathbb{C}_{+})$. Moreover
 $ \|\lambda\|_{\ell_{\alpha}^{\Phi_{1}}}^{lux} = 1.   $
 For $0 <s,t<1$ and   $z \in \mathbb{C}_{+}$, put
 $$  F_{t,s}(z) =\sum_{j\in \mathbb{Z}}\sum_{l\in \mathbb{Z}} \lambda_{l,j}r_{l}(t)r_j(s)F_{l,j}(z),   $$
 where 
 $$  F_{l,j}(z)= 2^{\alpha+2}\left(  \frac{z-\overline{z_{l,j}}}{i} \right)^{-\alpha-2}2^{j\gamma (\alpha +2)}.   $$
 By construction,  $F_{t,s} \in A_{\alpha}^{\Phi_{1}}(\mathbb{C}_{+})$ and $  \|F_{t,s}\|_{A^{\Phi_{1}}_{\alpha}}^{lux}\lesssim\|\lambda\|^{lux}_{\ell_{\alpha}^{\Phi_{1}}},   $  
 according to Theorem \ref{pro:mainqaaq6}. Since  $A_{\alpha}^{\Phi_{1}}(\mathbb{C}_{+})$ continuously embeds into  $L^{\Phi_{2}}(\mathbb{C}_{+}, d\mu)$, we obtain the following using Kinchine's inequality (Lemma \ref{pro:main2qapAaa}) and Fubbini's theorem.
 \begin{align*}
 \int_{\mathbb{C}_{+}}\Phi_{2}\left(\left(\sum_{l,j\in \mathbb{Z}}|\lambda_{l,j}|^{2}|F_{l,j}(z)|^{2}\right)^{1/2}\right)d\mu(z) &\lesssim \int_{\mathbb{C}_{+}}\int_{0}^{1}\int_{0}^{1}\Phi_{2}\left(\left|\sum_{l,j\in \mathbb{Z}}\lambda_{l,j}r_{l}(t)r_j(s)F_{l,j}(z)\right|\right)dtdsd\mu(z) \\
 &\lesssim \int_{0}^{1}\int_{0}^{1}\int_{\mathbb{C}_{+}}\Phi_{2}\left( \frac{|F_{t,s}(z)|}{\|F_{t,s}\|_{A^{\Phi_{1}}_{\alpha}}^{lux}} \right)d\mu(z)dtds 
 \lesssim 1. 
 \end{align*}
  We summarize the above as follows.
 \begin{equation}\label{eqoaqmapmq6}
 \int_{\mathbb{C}_{+}}\Phi_{2}\left(\left(\sum_{l,j\in \mathbb{Z}}|\lambda_{l,j}|^{2}|F_{l,j}(z)|^{2}\right)^{1/2}\right)d\mu(z) \lesssim 1.  
 \end{equation}
 For $t\geq 0$, put
  $$  \widetilde{\Phi}_{2}(t)=\Phi_{2}\left(t^{1/a_{\Phi_{2}}}\right).      $$
 By construction, $\widetilde{\Phi}_{2}$ is a convex growth function. For $z \in \mathbb{C}_{+}$, we have:
 \begin{align*}
 \sum_{l,j\in \mathbb{Z}}\Phi_{2}\left(|\lambda_{l,j}|\right)\chi_{\mathcal{D}_{\delta}\left(z_{l,j}\right)}(z)&\sim\sum_{l,j\in \mathbb{Z}}\Phi_{2}\left(|\lambda_{l,j}|\times \frac{2^{j\gamma (\alpha +2)}}{|z-\overline{z_{l,j}}|^{2+\alpha}}\right)\chi_{\mathcal{D}_{\delta}\left(z_{l,j}\right)}(z)\\
 &\sim
 \sum_{l,j\in \mathbb{Z}}\widetilde{\Phi}_{2}\left(\left(|\lambda_{l,j}||F_{l,j}(z)|\chi_{\mathcal{D}_{\delta}\left(z_{l,j}\right)}(z)\right)^{a_{\Phi_{2}}}\right)\\
 &\lesssim \widetilde{\Phi}_{2}\left(\sum_{l,j\in \mathbb{Z}}\left(|\lambda_{l,j}||F_{l,j}(z)|\chi_{\mathcal{D}_{\delta}\left(z_{l,j}\right)}(z)\right)^{a_{\Phi_{2}}}\right).
 \end{align*}
If $a_{\Phi_{2}} \geq 1$ then we have 
\begin{align*}
\sum_{l,j\in \mathbb{Z}}\left(|\lambda_{l,j}||F_{l,j}(z)|\chi_{\mathcal{D}_{\delta}\left(z_{l,j}\right)}(z)\right)^{a_{\Phi_{2}}}&\lesssim \left(\sum_{l,j\in \mathbb{Z}}\left(|\lambda_{l,j}||F_{l,j}(z)|\right)^{2a_{\Phi_{2}}}\right)^{\frac{1}{2}}\left(\sum_{l,j\in \mathbb{Z}}\left(\chi_{\mathcal{D}_{\delta}\left(z_{l,j}\right)}(z)\right)^{2a_{\Phi_{2}}}\right)^{\frac{1}{2}} \\
&=N^{1/2} \left(\sum_{l,j\in \mathbb{Z}}\left(|\lambda_{l,j}||F_{l,j}(z)|\right)^{2a_{\Phi_{2}}}\right)^{\frac{1}{2}} \\
&\lesssim N^{1/2} \left(\sum_{l,j\in \mathbb{Z}}\left(|\lambda_{l,j}||F_{l,j}(z)|\right)^{2}\right)^{\frac{a_{\Phi_{2}}}{2}}, 
\end{align*}
since each  $z \in \mathbb{C}_{+}$ belongs to at most $N$ of the sets $\mathcal{D}_{\delta}\left(z_{l,j}\right)$.
If $a_{\Phi_{2}} \leq 1$ then we have
\begin{align*}
\sum_{l,j\in \mathbb{Z}}\left(|\lambda_{l,j}||F_{l,j}(z)|\chi_{\mathcal{D}_{\delta}\left(z_{l,j}\right)}(z)\right)^{a_{\Phi_{2}}}&\lesssim \left(\sum_{l,j\in \mathbb{Z}}\left(|\lambda_{l,j}||F_{l,j}(z)|\right)^{a_{\Phi_{2}}\times \frac{2}{a_{\Phi_{2}}}}\right)^{\frac{a_{\Phi_{2}}}{2}}\left(\sum_{l,j\in \mathbb{Z}}\left(\chi_{\mathcal{D}_{\delta}\left(z_{l,j}\right)}(z)\right)^{a_{\Phi_{2}}\times \frac{1}{1-\frac{a_{\Phi_{2}}}{2}}}\right)^{1-\frac{a_{\Phi_{2}}}{2}} \\
&\lesssim N^{\left(1-\frac{a_{\Phi_{2}}}{2}\right)} \left(\sum_{l,j\in \mathbb{Z}}\left(|\lambda_{l,j}||F_{l,j}(z)|\right)^{2}\right)^{\frac{a_{\Phi_{2}}}{2}}. 
\end{align*}
We deduce that   
$$  \sum_{l,j\in \mathbb{Z}}\left(|\lambda_{l,j}||F_{l,j}(z)|\chi_{\mathcal{D}_{\delta}\left(z_{l,j}\right)}(z)\right)^{a_{\Phi_{2}}}\lesssim \max\left\{ N^{1/2}; N^{\left(1-\frac{a_{\Phi_{2}}}{2}\right)}    \right\} \left(\sum_{l,j\in \mathbb{Z}}\left(|\lambda_{l,j}||F_{l,j}(z)|\right)^{2}\right)^{\frac{a_{\Phi_{2}}}{2}}. 
   $$
 It follows that
 \begin{equation}\label{eqoamq6}
 \sum_{l,j\in \mathbb{Z}}\Phi_{2}\left(|\lambda_{l,j}|\right)\chi_{\mathcal{D}_{\delta}\left(z_{l,j}\right)}(z) \lesssim \Phi_{2}\left(\left(\sum_{l,j\in \mathbb{Z}}|\lambda_{l,j}|^{2}|F_{l,j}(z)|^{2}\right)^{1/2}\right).
 \end{equation}
 From (\ref{eqoaqmapmq6}) and (\ref{eqoamq6}), we obtain
 \begin{align*}
 |\langle    \{\widehat{\mu}_{\delta}(z_{l,j})\}; \{\nu_{l,j}\} \rangle_{\alpha} |&=
 \left|\sum_{l,j\in \mathbb{Z}}\widehat{\mu}_{\delta}(z_{l,j})\overline{\nu_{l,j}}2^{j\gamma (\alpha +2)} \right|\\
 &\lesssim      \sum_{l,j\in \mathbb{Z}}\Phi_{2}\left(|\lambda_{l,j}|\right)\widehat{\mu}_{\delta}(z_{l,j})2^{j\gamma (\alpha +2)} \\
 &\sim \sum_{l,j\in \mathbb{Z}}\Phi_{2}\left(|\lambda_{l,j}|\right)\mu(\mathcal{D}_{\delta}\left(z_{l,j}\right)) \\
 &\lesssim \int_{\mathbb{C}_{+}}\sum_{l,j\in \mathbb{Z}}\Phi_{2}\left(|\lambda_{l,j}|\right)\chi_{\mathcal{D}_{\delta}\left(z_{l,j}\right)}(z)d\mu(z)\\
 &\lesssim \int_{\mathbb{C}_{+}}\Phi_{2}\left(\left(\sum_{l,j\in \mathbb{Z}}|\lambda_{l,j}|^{2}|F_{l,j}(z)|^{2}\right)^{1/2}\right)d\mu(z) 
 \lesssim 1.
 \end{align*}
 We deduce that,  the sequence $ \{\widehat{\mu}_{\delta}(z_{l,j})\}_{l,j\in \mathbb{Z}} \in \ell_{\alpha}^{\Phi_{3}}(\mathbb{C}_{+})$. It follows that   $\widehat{\mu}_{\delta'} \in L^{\Phi_{3}}(\mathbb{C_{+}}, dV_{\alpha})$, for all sufficiently small  $\delta'$,  according to the Lemma \ref{pro:main1qpapqp8}. We also conclude that $\widetilde{\mu} \in L^{\Phi_{3}}(\mathbb{C_{+}}, dV_{\alpha})$, thanks to Lemma \ref{pro:main1qppaq8}.
 \epf

\section{Conclusion and perspectives.}

In this paper we have provided an atomic decomposition for Bergman–Orlicz spaces on the upper complex half-plane and used it to characterize Carleson embeddings with loss and the boundedness of composition operators in the loss range. The application of the atomic decomposition to Carleson embeddings with loss closes the gap left open in the earlier works \cite{djesehb,djesehafeuto}.

 \bibliographystyle{plain}
  
 \end{document}